\documentclass[11pt]{article}
\usepackage{graphicx,graphics}

\usepackage{latexsym}
\usepackage{caption}
\usepackage{amsfonts}
\usepackage{amsmath,amsthm,enumerate, mathrsfs, mathabx, amssymb, indentfirst}

\usepackage{enumerate}
\usepackage{enumitem}
\usepackage[normalem]{ulem}
\usepackage{authblk}
\usepackage{pict2e} 
\usepackage{tikz}
\usetikzlibrary{arrows.meta}

\usepackage{subfig}
\usepackage{pgfplots}
\pgfplotsset{compat=1.15}
\usepgfplotslibrary{colormaps, external}
\usepackage{float}
\usepackage{tikz}
\usetikzlibrary{calc}
\usetikzlibrary{backgrounds, shapes.geometric}
\usetikzlibrary{arrows.meta,decorations.markings}
\usepackage{rotating}
\usetikzlibrary{positioning}
\usepackage{placeins}

\usepackage{pgffor}
\usepackage{tikzit}
\usepackage{tikz-cd}

\tikzstyle{Vertex}=[fill={rgb,255: red,142; green,142; blue,142}, draw=black, shape=rectangle]
\tikzstyle{white vertex}=[fill=white, draw=black, shape=circle]
\tikzstyle{grey vertex}=[fill={rgb,255: red,142; green,142; blue,142}, draw=black, shape=circle]
\tikzstyle{black vertex}=[fill=black, draw=black, shape=circle]
\tikzstyle{small black}=[fill=black, draw=black, shape=circle, scale=.5pt, tikzit draw=white]
\tikzstyle{black empty}=[fill=white, draw=black, shape=circle]
\tikzstyle{small black empty}=[fill=white, draw=black, shape=circle, scale=.5pt, tikzit draw={rgb,255: red,156; green,156; blue,156}, tikzit fill=white]

\tikzstyle{double blue edge}=[->, draw={rgb,255: red,6; green,118; blue,255}, latex-latex]
\tikzstyle{automorphism}=[->, -latex, draw=red]
\tikzstyle{new edge style 0}=[->, draw={rgb,255: red,6; green,118; blue,255}, -latex]
\tikzstyle{Edge}=[-]
\tikzstyle{Arc}=[->, -latex]
\tikzstyle{blue_edge}=[-, draw=blue, fill=none]
\tikzstyle{edge red}=[-, fill=none, draw=red]
\tikzstyle{green edge}=[-, draw={rgb,255: red,0; green,168; blue,0}, -latex]
\tikzstyle{big arc}=[->, line width=1.2pt, -latex]
\tikzstyle{new edge style 1}=[-, fill={rgb,255: red,191; green,191; blue,191}, opacity=0.5]
\tikzstyle{big edge}=[-, line width=1.2pt]
\tikzstyle{dashed_arc}=[->, -latex, dashed]

\usepackage{thmtools}
\usepackage{cleveref}

\usepackage{tikz-3dplot}
\pgfplotsset{compat=1.15}

\newtheorem{theorem}{Theorem}

\newtheorem{lemma}[theorem]{Lemma}
\newtheorem{corollary}[theorem]{Corollary}
\newtheorem{observation}[theorem]{Observation}
\newtheorem{proposition}[theorem]{Proposition}

\newtheorem{problem}[theorem]{Problem}

\renewcommand{\cal}[1]{\mathcal{#1}}
\newcommand{\C}[1]{\mathcal{C}}

\newcommand{\set}[1]{\left\{#1\right\}}
\newcommand{\curve}[1]{\langle#1\rangle}
\newcommand{\clock}{\mathbin{\tikz[baseline=-0.25ex, scale=0.9]{
  \draw[
    line width=0.6pt,
    -{Stealth[length=1.1mm, width=1mm]}
  ]
  (0,0) arc[start angle=-30, end angle=-335, radius=0.65ex]
  -- ++(305:0.45ex);
}}}

\usetikzlibrary{calc}
\tikzset{
	-parallel segment/.style={
		segment distance/.store in=\segDistance,
		segment pos/.store in=\segPos,
		segment length/.store in=\segLength,
		to path={
			($(\tikztostart)!\segPos!(\tikztotarget)!\segLength/2!(\tikztostart)!\segDistance!90:(\tikztotarget)$) -- 
			($(\tikztostart)!\segPos!(\tikztotarget)!\segLength/2!(\tikztotarget)!\segDistance!-90:(\tikztostart)$)  \tikztonodes
		}, 
		segment pos=.5,
		segment length=6ex,
		segment distance=1mm,
	}}

\tikzset{
	parallel segment/.style={
		segment distance/.store in=\segDistance,
		segment pos/.store in=\segPos,
		segment length/.store in=\segLength,
		to path={
			($(\tikztostart)!\segPos!(\tikztotarget)!\segLength/2!(\tikztostart)!\segDistance!-90:(\tikztotarget)$) -- 
			($(\tikztostart)!\segPos!(\tikztotarget)!\segLength/2!(\tikztotarget)!\segDistance!90:(\tikztostart)$)  \tikztonodes
		}, 
		segment pos=.5,
		segment length=6ex,
		segment distance=1mm,
}}

\DeclareFontEncoding{LS1}{}{}
\DeclareFontSubstitution{LS1}{stix}{m}{n}
\DeclareSymbolFont{stixletters}{LS1}{stix}{m}{it}
\DeclareMathAccent{\cev}{\mathord}{stixletters}{"91}
\DeclareMathAccent{\vec}{\mathord}{stixletters}{"92}

\usepackage{todonotes}

\definecolor{lightblue}{rgb}{0.68,0.85,0.9}

\begin{document}
	\setlength{\parindent}{0pt}	
	
	\title{Winding number and circular coloring}	
	\author[1,2]{Reza Naserasr}
	\author[1]{Cyril Pujol}
	\author[3]{Lujia Wang}
	
	\affil[1]{\small Université Paris Cité, CNRS, IRIF, F-75013, Paris, France. {Emails: \texttt{\{reza, cpujol\}@irif.fr}}}
	\affil[2]{\small Zhejiang Normal University, Jinhua, China. }
	\affil[3] {\small Hainan Bielefeld University of Applied Sciences, Danzhou, China. {Email: \texttt{lujia.wang@hibiuh.edu.cn}}}
	
	\date{\today} 
	
\maketitle 
\begin{abstract}
In 1996, Youngs proved a surprising theorem that quadrangulations of the projective plane could never have chromatic number exactly 3. This sparked a lot of interest, and  the result has been further developed in many directions over the past decades. For example, the result is strengthened by considering the circular chromatic number, which is a real-valued lower bound on the chromatic number: The circular chromatic number of a quadrangulation cannot be in the interval $(2,4)$. This parameter allows a generalization to larger even faces, for which a similar gap exists.

In this work, we place these results into a framework based on the notion of winding number using extensions of colorings to continuous mappings.This yields unified and simplified proofs of gaps in the circular chromatic number for graphs with a distinguished set of directed even cycles. This generalizes the setting of graphs embedded on surfaces where every face is even.

We further establish an analogous gap phenomenon when all faces are of a given odd length, previously known only in the case of triangulations. For example, we conclude that if $G$ is a graph embedded on the projective plane such that all faces are 5-cycles, then either its circular chromatic number is  $\frac{5}{2}$ or at least $3$, the former being the case only if $G$ is Eulerian and every noncontractible facial walk is of odd length.

We then extend this framework to signed graphs, where the analogy between vertex switching and $\mathbb{Z}_2$-actions provides a natural topological setting. We show that given a circular $r$-coloring $\phi$ of a signed graph, where $r$ is relatively small, a specific continuous extension of $\phi$ over the signed graph, viewed as a topological object, must satisfy certain requirements. This translates to lower bounds on the circular chromatic number of certain families of signed graphs. 

In particular, we show that the circular chromatic number of a signed graph embedded on the projective plane whose faces all have length $k$ cannot be in the interval $(2,\frac{2k}{k-2})$ if the faces are positive, nor in the interval $(\frac{2k}{k-1},\frac{2k-2}{k-2})$ if the faces are negative.

Applying this on signed triangulations of the projective plane where facial cycles are all positive, we conclude that any such signed graph is either balanced or has circular chromatic number 6. It is observed that this easily implies the aforementioned result for quadrangulations. For any positive integer $k$, we construct a triangulation of the projective plane where each face is a positive triangle and the negative girth is $k$. This is in contrast with the (unsigned) graph case since $K_6$ is the only 6-critical graph embedded on the projective plane and, moreover, C. Thomassen proved that there are only finitely many 6-critical graphs on any given surface. 

\end{abstract}

\section{Introduction}
	
	In the 1990s, several results of similar nature were produced by independent groups of researchers.  
	Extending topological methods, Stiebitz, after introducing generalized Mycielskian (also known as the cone cover), proved in \cite{S85} that if a graph is obtained from $K_2$ by repeatedly applying generalized Mycielski construction, then at each step its chromatic number is increased by 1. Although the original result is not accessible, the key ideas appear in \cite{GYJS}. Studying 4-chromatic graphs of large odd girth, Van Ngoc and Tuza introduced $M_k(C_{2k+1})$ as potentially the smallest 4-chromatic graph of odd girth $2k+1$, this remains an open problem until now. Studying quadrangulations of the projective plane, Youngs proved in \cite{Y96} that such graphs are never 3-chromatic. Prior to Youngs result, Payan, in \cite{P92} proved a similar statement that cube-like graphs (Cayley graphs on binary groups) are never 3-chromatic. We refer to \Cref{sec:Examples} for details on these families, but we note that $M_k(C_{2k+1})$ can be embedded on the projective plane as a quadrangulation. They are also the 4-chromatic graphs that Payan identified in \cite{P92} as 4-chromatic subgraphs of non-bipartite cube-like graphs (see also \cite{BNT15}). These results have been strengthened and extended in numerous ways in the past 3 decades from various points of view. Among others, results of \cite{MS02} and \cite{DGMVZ05} are of special interest as we further extend their view points.	
	
	Here we put these families of results in a larger setting and among other things, we propose extensions to signed graphs. Our follow up work shows that the application of algebraic topology to graph coloring is in fact better suited for extended notions of coloring based on the language of signed graphs. That is because of the analogy between $\mathbb{Z}_2$-actions on topological spaces (that are essential elements of connection to coloring) and vertex switching operation in signed graphs.
	
	For standard notion of graph theory we follow \cite{W96}. We settle our choice of terminology for signed graphs in the following subsection, referring to \cite{NRS15, NSZ21} for more. Formalizing the notion of winding number is done in \Cref{sec:winding}. In \Cref{sec:GeneralSetting} we study general properties of winding number for graphs. We apply the development to prove lower bounds on the circular chromatic number of graphs in \Cref{sec:Circular-Graphs}. Winding number for signed graphs together with application on the circular chromatic number is considered in \Cref{sec:Widing-Circular-Signed}. Special classes of signed graphs, which are candidate to be the smallest signed graphs of the given negative girth and circular chromatic number, are presented in \Cref{sec:Examples}.

	\subsection{Signed graphs}\label{sec:SignedGraphs}
	
	A signed graph $(G, \sigma)$ is a graph $G$ together with an assignment $\sigma$ of signs ($+$ or $-$) to its edges. The sign of a closed walk in $(G, \sigma)$ is the product of signs of its edges considering multiplicity. A signed graph (or a subset of its vertices) is said to be \emph{balanced} if it induces no negative cycle. Switching a vertex $v$ means multiplying the sign of each incident edge by a $-$. A key property in the study of signed graphs is that switching a vertex does not impact the signs of closed walks and cycle. The balanced chromatic number of a signed graph $(G,\sigma)$, denoted $\chi_b(G,\sigma)$, is the minimum number of balanced sets which covers $V(G)$.

	\subsection{Winding number}\label{sec:winding}
	
	The \emph{winding number} is an algebraic notion used in topology in order to distinguish closed curves that cannot continuously be transformed to one another. In intuitive terms, for a continuous function from $S^1$ to $S^1$ the winding number counts how many times the first, as a curve, ``winds around the circle''. For example in \Cref{fig:winding_number} the blue curve winds around the black circle twice. Depending on whether the wrapping around the circle is done in the clockwise direction or counterclockwise, the winding number is positive or negative, noting that the opposite directions cancel each other out. 	
	
\begin{figure}[!htb]
		\centering
		\scalebox{2}{\begin{tikzpicture}[scale=2]
	\draw[very thick] (0,0) circle (1);
	\def\radius{1.15+0.07*\t+0.1*(atan(40*(\t-0.45))/90))}
	\def\angle{100 + 720*\t - 80*exp(-800*(\t-0.45)^2)}

	\draw[thick,blue,samples=200,domain=0:1,smooth,variable=\t]
	plot ({\radius*-cos(\angle)},
	      {\radius*sin(\angle)});
\end{tikzpicture}}        
		\caption{Example of a curve with winding number 2}
		\label{fig:winding_number}
	\end{figure}
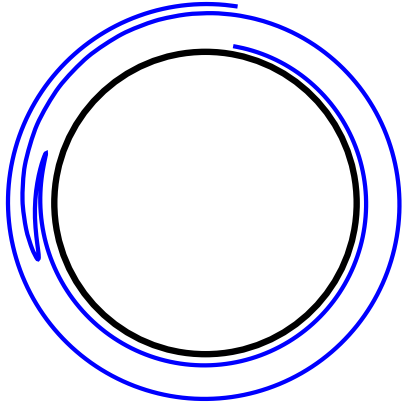

	The application of algebraic topology for providing lower bounds on the chromatic number of graphs satisfying certain topological properties has become an integral part of the theory of graph coloring since Lov\'asz' proof of the Kneser conjecture \cite{L78}. The simpler notion of winding number implicitly appears first (to our knowledge) in El-Zaher and Sauer's proof of Hedetniemi's conjecture for 4-chromatic graphs \cite{ES85}.
	
	Some of the proofs of the results mentioned above trace back to the winding number as well. In this work, we formalize the notion of winding number for the purpose of proving lower bounds for the circular chromatic number of graphs and signed graphs.  The approach has several benefits: first a better understanding of what makes the lower bounds work, secondly the resulting criteria applies to larger classes of graphs, third we extend the result to signed graphs. 
	
	\section{General setting for graphs}\label{sec:GeneralSetting}
	
	Throughout this paper a cycle $C$ is considered with a cyclic ordering of the vertices. This ordering corresponds to one of the two directions on the cycle. Direction is a specific orientation of $C$ where each vertex is tail of one arc and head of the other. The term orientation is also used on geometric objects such as a circle or a closed curve referring to one of the two possible ways the curve is traversed. A directed cycle embedded in the plane inherits term `clockwise' or `counterclockwise' from the geometric point of view.

	With $\mathcal{O}$ we denote a (geometric) circle on the plane which is endowed with an orientation referred to as the clockwise orientation. We normally assume the center is the origin and hence use $-x$ to denote the antipode of $x$ when $x$ is a point on $\mathcal{O}$. When the circumference of $\mathcal{O}$ is important, assuming that it is $r$, we will write $\mathcal{O}_r$ instead. Given a cycle $C$ and a mapping $\phi$ of its vertices to the points on $\mathcal{O}$, we consider linear extensions of $\phi$ to mappings of $C$ to $\mathcal{O}$ which linearly maps the edge $x_{i}x_{i+1}$ to one of the two arcs of the circle $\mathcal{O}$ determined by the points $\phi(x_{i})$ and $\phi(x_{i+1})$. 
	Assuming $C$ is a cycle of length $k$, it follows that in total there are $2^k$ linear extensions of $\phi$ as mappings of $C$ to $\mathcal{O}$. We employ the following notion of edge-coloring to refer to these extensions. Observe that subject to rotation of the circle, for a connected graph $G$, the mapping $\phi$ of vertices to $\mathcal{O}$ is also determined by deciding for each edge the length it must travel in the given direction for that edge.
	
	Let $I$ be a small interval on $\mathcal{O}$ which contains no element of $\phi(V(C))$. For most purposes we may assume $I$ consists of a single point. Linear extension of an edge $x_{i}x_{i+1}$ either fully contains $I$ or has no intersection with it. We associate a one-to-one correspondence between all such extensions and green-orange coloring of the edges of $C$ by associating the cases that $\phi(x_{i})\phi(x_{i+1})$ contains $I$ with green and the cases that do not contain it with orange. This is clearly a one-to-one correspondence between linear extensions of $(C, \phi)$ and colorings of the edges of $G$ with two colors. Assuming $\rho$ is such a coloring, we denote the associated closed curve with $\curve{C, \phi, \rho}$, but each of the two functions might get dropped from the notation when it is clear from the context. The winding number of $\curve{C, \phi, \rho}$ can be computed as the number of green edges that traverse $I$ in the clockwise direction minus the number of green edges that traverse $I$ in the counterclockwise direction. For the parity of the winding number we only need to consider the total number of green edges in which case $I$ could be considered as a single point.

	As winding number is the function of the number of green arcs (noting that direction matters), if two sets of cycles induce the same set of green arcs, then the parity of their total winding numbers are the same. This is the key idea of this work: given a set of cycles, under certain assumptions, we show that after choosing directions and then reshuffling the edges into a different set of cycles we get a winding number whose parity could not be the same as that of the original one. Simplest cycles to use in this reshuffling process are 2-cycles about which we have the following basic observation.

	\begin{observation}\label{obs:2-cycle}
		A (directed) 2-cycle $C$ with one green and one orange edge has winding number $\pm 1$. 
	\end{observation}
	  
	 Among $2^k$ possible linear extensions of $\phi$ two are of special interest for us (illustrated in \Cref{fib:clockShortExtension}): 
	 
	 \begin{itemize}
	 	
	 	\item When each edge $x_{i}x_{i+1}$ is mapped to $\phi(x_{i})\phi(x_{i+1})$ using the part of $\mathcal{O}$ in clockwise direction from $\phi(x_{i})$ to $\phi(x_{i+1})$, noting that $x_{n+1}=x_1$. This is well defined as long as $\phi(x_i) \neq \phi(x_{i+1})$.The winding number of this extension is denoted by $w\curve{\vec{C},\phi,\clock}$.
	 		 	
	 	\item When $x_{i}x_{i+1}$ is mapped to the shorter of the two arcs whose end points are $\phi(x_{i})$ and $\phi(x_{i+1})$. This is well defined as long as $\phi(x_i)$ and $\phi(x_{i+1})$ are not antipodal. The winding number of $(\vec C, \phi)$ in this convention is denoted by $w\curve{C, \phi,sh}$.
	 \end{itemize}

\begin{figure}[h]
\centering
\begin{tikzpicture}[scale=1.]
\def\r{1}

\node[fill=white, draw=black, shape=circle,label=above:$x$]   (xnode)  at (0,0) {};
\node[fill=white, draw=black, shape=circle,label=above:$y$]   (ynode)  at (2,0) {};
\draw (xnode) -- (ynode);

\draw[-{Stealth[length=2mm]}] (0.9,0) -- ++ (0.2,0);   

\coordinate (C) at (4,0);
\draw (C) circle (\r);
\node[circle,fill,inner sep=1pt,label=above:$\phi(x)$]   (x)  at ($(C)+(90:\r)$) {};
\node[circle,fill,inner sep=1pt,label={[xshift=-3pt,yshift=2pt]above:$\phi(y)$}]  (y)  at ($(C)+(150:\r)$) {};
\node[fill,inner sep=1pt,label=right:$I$]  (i)  at ($(C)+(0:\r)$) {};

\draw[orange,line width=.8mm] (x) arc[start angle=90,end angle=150,radius=1];

\coordinate (D) at (7,0);
\draw (D) circle (\r);
\node[circle,fill,inner sep=1pt,label=above:$\phi(x)$]   (x)  at ($(D)+(90:\r)$) {};
\node[circle,fill,inner sep=1pt,label={[xshift=-3pt,yshift=2pt]above:$\phi(y)$}]  (y)  at ($(D)+(150:\r)$) {};

\draw[black!30!green,line width=.8mm] (x) arc[start angle=90,end angle=150-360,radius=1];
\node[fill,inner sep=1pt,label=right:$I$]  (i)  at ($(D)+(0:\r)$) {};

\end{tikzpicture}
\caption{Shortest (left) and clockwise (right) extension of an edge $xy$}
\label{fib:clockShortExtension}
\end{figure}
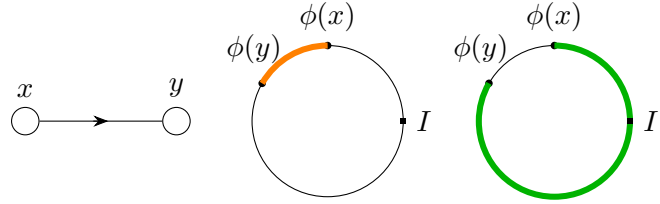
	 
	 It follows from the definition that:
	 
	 \begin{observation}\label{obs:Clockwise+CounterCW}
	 	In the clockwise (or counterclockwise) extension $\curve{\vec{C},\phi,\clock}$ the winding number is equal to the number of edges colored green.
	 \end{observation}
	 
	 The next basic but important observation is about the impact of changing the direction of the cycle $\vec{C}$. Let $\cev{C}$ be the directed cycle obtained from $\vec{C}$ by flipping its direction. 
	   
	 \begin{proposition}\label{prop:OpposingDirections}	 		
	 	For any directed cycle $\vec C$ and any mapping $\phi$ of its vertices to the circle $\mathcal{O}$ we have $-w\curve{\cev{C},\phi,\clock}+w\curve{\vec{C},\phi,\clock}=|C|$.
	 \end{proposition}

	 \begin{proof}
	 	That is because, with a fixed mapping $\phi$ of the vertices of $C$ to $\mathcal{O}$, an edge $x_{i}x_{i+1}$ is colored green in $\curve{\vec{C}, \phi, \clock}$ if and only if it is colored orange in $\curve{\cev{C}, \phi, \clock}$. The rest follows from \Cref{obs:2-cycle} and \Cref{obs:Clockwise+CounterCW}.
 	 \end{proof}
 	 
 	 The next crucial observation is that if the mapping $\phi$ of \Cref{prop:OpposingDirections} satisfies certain structural conditions, then two values of $w\curve{\vec{C},\clock}$ and $w\curve{\cev{C},\clock}$ are as close as possible. This is where the condition of $\phi$ being a circular coloring kicks in. The precise statement is given after recalling the definition of circular coloring.
 	 
 	 A mapping $\phi$ of $V(G)$ to $\mathcal{O}_r$ is a said to be a \emph{circular $r$-coloring} if for each edge $xy$ of $G$ each of the two arcs of $\mathcal{O}_r$ with end points $\phi(x)$ and $\phi(y)$ is of length at least 1. The assignment of direction and corresponding length to the edges would be called an \emph{$r$-tension}. The condition on the length of travel for each edge leads to the following.
 	 
 	 \begin{lemma}
 	 	If $\phi$ is a circular $r$-coloring of a directed cycle of length $k$, then $w\curve{\vec{C},\clock}\geq \lceil \frac{k}{r} \rceil $. 
 	 \end{lemma}
 	 
 	 \begin{proof}
 	 	That is simply because each edge $x_{i}x_{i+1}$ of the cycle must travel in the clockwise direction a distance of at least 1 and that the winding number is an integer.  
 	 \end{proof}
 	 
 	 Similarly, $w\curve{\cev{C},\clock}\geq \lceil \frac{k}{r} \rceil$. Thus if $r$ is such that $\lceil \frac{k}{r} \rceil \geq \frac{k}{2}$, then, by  \Cref{prop:OpposingDirections}, the values of the pair $w\curve{\vec{C},\clock}, w\curve{\cev{C},\clock}$ are determined. This is made precise based on the parity of $|C|$ in the following two lemmas.

 	 \begin{proposition}\label{prop:2k-cycle}
 	 	If $r<\frac{2k}{k-1}$ and $\phi$ is a circular $r$-coloring of a $2k$-cycle $C_{2k}$, then $$-w\curve{\cev{C}_{2k}, \phi, \clock}=w\curve{\vec{C}_{2k}, \phi, \clock}=k.$$
 	 \end{proposition}
 	 
 	 In other words, in the green-orange coloring of the edges corresponding to these two extensions of $C$, exactly half the edges are green in $\vec{C}$ and the other half in $\cev{C}$.

 	  \begin{proposition}\label{prop:2k+1-cycle-I}
 	 	If $r<\frac{2k+1}{k-1}$ and $\phi$ is a circular $r$-coloring of a directed $(2k+1)$-cycle $C_{2k+1}$, then $$\{-w\curve{\cev{C}_{2k+1}, \phi, \clock}, w\curve{\vec{C}_{2k+1}, \phi, \clock}\}=\{k, k+1\}.$$
 	 \end{proposition}

 	 As will be observed, the conditions of \Cref{prop:2k-cycle} are strong enough to prove lower bounds on the circular chromatic number of graphs embedded on the projective plane where each face is an even cycle of prescribed length. This is a direction widely studied in the literature. However, the conditions of \Cref{prop:2k+1-cycle-I} are not strong enough for results of similar nature. For this we need the following stronger conclusion which requires further restriction on $r$. This observation allows us to extend one of the classical results in graph coloring first proved in \cite{H1898}: a planar triangulation is 3-colorable if and only if it is Eulerian. 
 	 	
      \begin{proposition}\label{prop:2k+1-cycle-II}
     	If $r<\frac{4k}{2k-1}$ and $\phi$ is a circular $r$-coloring of a directed $(2k+1)$-cycle $C_{2k+1}$, then $\{-w\curve{\cev{C}_{2k+1}, \phi, \clock}, w\curve{\vec{C}_{2k+1}, \phi, \clock}\}=\{k, k+1\}$. Furthermore, assuming $ w\curve{\vec{C}_{2k+1}, \phi, \clock}=k$ each edge of $\vec{C}_{2k+1}$ traverses a clockwise distance of less than $\frac{r}{2}$.  
     \end{proposition}	
	
	\begin{proof}
		Since the condition on $r$ is more restrictive than $r<\frac{2k+1}{k-1}$, the first part follows from \Cref{prop:2k+1-cycle-I}. For the furthermore part, assume to the contrary that one edge traverses a distance of $\frac{r}{2}$ or more. Then even if each of the other edges traverses the minimum required clockwise distance of $1$, the total travel length will be at least $2k+\frac{r}{2}$. Combined with the condition of $r<\frac{4k}{2k-1}$, this implies a winding number of at least $k+1$.
	\end{proof}
	
	In other words, if $\phi$ is a circular $r$-coloring of a $(2k+1)$-cycle with $r<\frac{4k}{2k-1}$, then either clockwise extension of $\phi$ or counterclockwise extension of $\phi$ coincides with the shortest extension. This leads to the following corollary which presents our main use of this proposition.
	
	\begin{corollary}\label{cor:OddCyclesCommonEdge}
		Let $C$ and $C'$ be two $(2k+1)$-cycles sharing an edge $e$ and assume $\phi$ is circular $r$-coloring of the subgraph induced by the two where $r<\frac{4k}{2k-1}$. Then if we choose a direction for each of $C$ and $C'$  in such a way that the clockwise extension of $\phi$ will have winding number $k$, then the common edge $e$ will be oriented the same on both of them.
	\end{corollary}

	We are now ready to present structural conditions on a set of cycles which provides lower bounds on the circular chromatic of a graphs. We present the conditions separately for even and odd cycles.
	
	\section{Lower bound on circular chromatic number of graphs}\label{sec:Circular-Graphs}

	Let $\mathcal{C}$ be a set of cycles of a graph $G$ and let $D$ be an assignment of a direction to each cycle in $\mathcal{C}$. For each edge $xy$ of $G$ we define its $(\mathcal{C}, D)$-degree to be the difference of the number of times it is oriented $x\to y$ and $y\to x$. We are specially interested in cycle systems where the degree of each edge is even. Observe that this condition is independent from the choice of directions $D$. Here we provide lower bounds for the circular chromatic number of graphs provided they possess cycle systems of certain type. Then we show how to apply the results on families of graphs with special embedding on surfaces. But before going further we need to extend the classical notion of dual. 
	
	\subsection{Dual}  
	
	Given a system $(\mathcal{C}, D)$ of a set of cycles of a graph $G$ together with directions $D$, we define the \emph{dual} of the system to be a signed multigraph $S(\mathcal{C}, D)$ as follows. 
	
	\begin{itemize}
		\item The cycles in $\mathcal{C}$ are the vertices of the dual.
		\item A pair, $C_1$ and $C_2$, of cycles sharing an edge $uv$ are connected by an edge labeled $uv$ whose sign is determined by the following rules.
		\item If under the directions assigned by $D$ to $C_1$ and $C_2$ the edge $uv$ is directed in the same direction, then the associated edge in the dual is positive. If they are assigned opposing direction, then the associated edge is negative.  
	\end{itemize} 
	
	Observe that this extends the classic notion of the ``dual'' for planar graphs. To embed coloring notions considered in this work in the larger framework of signed graphs, we may consider graphs as signed graphs where all edges are negative. Consider a (2-connected) planar graph together with a cycle system consisting of facial cycles, each with clockwise orientation. Then each edge is in two faces and is oriented in the opposite directions on the two faces. Hence each edge in the dual signed graph of this system is negative. That is the classic dual of the planar graph endowed with a fully negative signature. 
	
	\begin{observation}\label{obs:switch_dual}
		Changing the direction of a cycle in the system $(\cal C,D)$ corresponds to switching the corresponding vertex in the dual.
	\end{observation}
	
	Therefore, when the directions are of little importance, we speak of the dual of a set of cycles denoted $\widehat S(\cal C)$.
		
	\subsection{Even cycles}
	
	\begin{theorem}\label{thm:setOfEvenCycles}
		Let $\mathcal{C}$ be a set of even cycles of a graph $G$ each of length at most $2k$ where each edge is in an even number of cycles in $\mathcal{C}$. If for some assignment $D$ of directions to the cycles in $\mathcal{C}$, the sum of the  $(\mathcal{C}, D)$-degrees of the edges is $2 \pmod 4$, then $G$ has circular chromatic number at least $\frac{2k}{k-1}$. 
	\end{theorem}

	\begin{proof}
		We consider two oriented multigraphs $G_1$ and $G_2$. The first is obtained from $\mathcal{C}$ by taking their edge disjoint union but allowing the same vertices to be identified. The second, $G_2$ is built from $G_1$ by adding a (minimum) number of edges so that each edge $xy$ is oriented as many times as $x\to y$ as it is oriented as $y\to x$. 
		Since the degree of each edge is even, the number of edges added for each $xy$ is even. But the assumption on the total degree implies that the total number of edges added to $G_1$ in order to get $G_2$ is of the form $2(2l+1)$ for some $l$. 
		
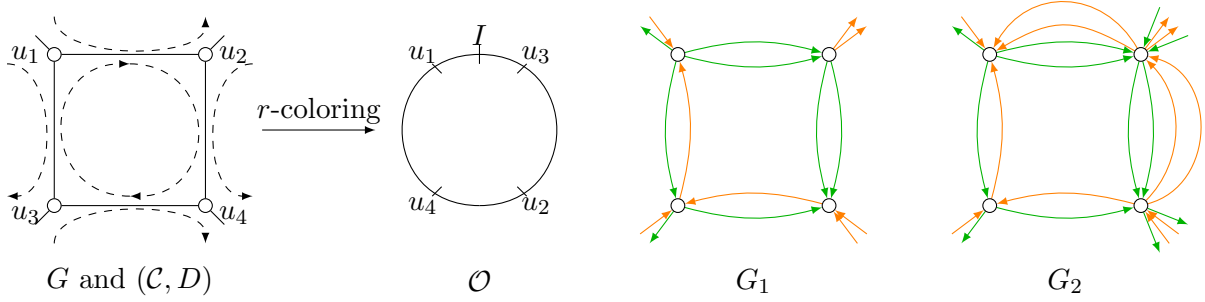
\begin{figure}[h]
	\centering
	\scalebox{1}{\begin{tikzpicture}
	\begin{pgfonlayer}{nodelayer}
		\node [style=small black empty] (0) at (-13.75, 4) {};
		\node [style=small black empty] (1) at (-9.75, 4) {};
		\node [style=small black empty] (2) at (-13.75, 0) {};
		\node [style=small black empty] (3) at (-9.75, 0) {};
		\node [style=none] (4) at (-11.75, 3.75) {};
		\node [style=none] (5) at (-11.75, 0.25) {};
		\node [style=none] (6) at (-15, 3.75) {};
		\node [style=none] (7) at (-15, 0.25) {};
		\node [style=none] (8) at (-8.5, 3.75) {};
		\node [style=none] (9) at (-8.5, 0.25) {};
		\node [style=none] (10) at (-13.75, 5) {};
		\node [style=none] (11) at (-9.75, 5) {};
		\node [style=none] (12) at (-13.75, -1) {};
		\node [style=none] (13) at (-9.75, -1) {};
		\node [style=small black empty] (14) at (2.75, 4) {};
		\node [style=small black empty] (15) at (6.75, 4) {};
		\node [style=small black empty] (16) at (2.75, 0) {};
		\node [style=small black empty] (17) at (6.75, 0) {};
		\node [style=none] (18) at (-11.75, -2) {$G$ and $(\mathcal C,D)$};
		\node [style=none] (19) at (4.75, -2) {$G_1$};
		\node [style=small black empty] (20) at (11, 4) {};
		\node [style=small black empty] (21) at (15, 4) {};
		\node [style=small black empty] (22) at (11, 0) {};
		\node [style=small black empty] (23) at (15, 0) {};
		\node [style=none] (24) at (13, -2) {$G_2$};
		\node [style=none] (25) at (-2.5, 4) {};
		\node [style=none] (26) at (-2.5, 0) {};
		\node [style=none] (31) at (-14.5, 4) {$u_1$};
		\node [style=none] (32) at (-9, 4) {$u_2$};
		\node [style=none] (33) at (-14.5, -0.25) {$u_3$};
		\node [style=none] (34) at (-9, -0.25) {$u_4$};
		\node [style=none] (35) at (-4, 4) {$u_1$};
		\node [style=none] (36) at (-1, 4) {$u_3$};
		\node [style=none] (37) at (-4, 0) {$u_4$};
		\node [style=none] (38) at (-1, 0) {$u_2$};
		\node [style=none] (39) at (-2.5, 4.5) {$I$};
		\node [style=none] (40) at (-2.5, 3.75) {};
		\node [style=none] (41) at (-2.5, 4.25) {};
		\node [style=none] (42) at (-2.5, -2) {$\mathcal O$};
		\node [style=none] (43) at (-8.25, 2) {};
		\node [style=none] (44) at (-5.25, 2) {};
		\node [style=none] (46) at (-6.75, 2.5) {$r$-coloring};
		\node [style=none] (47) at (-3.75, 3.75) {};
		\node [style=none] (48) at (-3.5, 3.5) {};
		\node [style=none] (49) at (-1.25, 3.75) {};
		\node [style=none] (50) at (-1.5, 3.5) {};
		\node [style=none] (51) at (-3.75, 0.25) {};
		\node [style=none] (52) at (-3.5, 0.5) {};
		\node [style=none] (53) at (-1.5, 0.5) {};
		\node [style=none] (54) at (-1.25, 0.25) {};
		\node [style=none] (55) at (1.75, 4.75) {};
		\node [style=none] (56) at (7.5, 5) {};
		\node [style=none] (57) at (1.75, -0.75) {};
		\node [style=none] (58) at (7.75, -0.75) {};
		\node [style=none] (59) at (2, 5) {};
		\node [style=none] (60) at (7.75, 4.75) {};
		\node [style=none] (61) at (2, -1) {};
		\node [style=none] (62) at (7.5, -1) {};
		\node [style=none] (63) at (-9.25, 4.5) {};
		\node [style=none] (64) at (-14.25, 4.5) {};
		\node [style=none] (65) at (-14.25, -0.5) {};
		\node [style=none] (66) at (-9.25, -0.5) {};
		\node [style=none] (67) at (10.25, 5) {};
		\node [style=none] (68) at (10, 4.75) {};
		\node [style=none] (69) at (10, -0.75) {};
		\node [style=none] (70) at (10.25, -1) {};
		\node [style=none] (71) at (15.75, -1) {};
		\node [style=none] (72) at (16, -0.75) {};
		\node [style=none] (73) at (15.75, 5) {};
		\node [style=none] (74) at (16, 4.75) {};
		\node [style=none] (75) at (15.5, 5.25) {};
		\node [style=none] (76) at (16.25, 4.5) {};
		\node [style=none] (77) at (16.25, -0.5) {};
		\node [style=none] (78) at (15.5, -1.25) {};
	\end{pgfonlayer}
	\begin{pgfonlayer}{edgelayer}
		\draw (0) to (1);
		\draw (1) to (3);
		\draw (3) to (2);
		\draw (2) to (0);
		\draw [style={dashed_arc}, bend left=90, looseness=1.75] (4.center) to (5.center);
		\draw [style={dashed_arc}, bend left=90, looseness=1.75] (5.center) to (4.center);
		\draw [style={dashed_arc}, bend left=90] (6.center) to (7.center);
		\draw [style={dashed_arc}, bend right=90] (8.center) to (9.center);
		\draw [style={dashed_arc}, bend right=90, looseness=0.75] (10.center) to (11.center);
		\draw [style={dashed_arc}, bend left=90, looseness=0.75] (12.center) to (13.center);
		\draw [style=Arc, draw={black!30!green}, bend right=15] (14) to (15);
		\draw [style=Arc, draw={black!30!green}, in=165, out=15] (14) to (15);
		\draw [style=Arc, draw=orange, bend right=15] (16) to (14);
		\draw [style=Arc, draw={black!30!green}, bend right=15] (14) to (16);
		\draw [style=Arc, draw={black!30!green}, bend right=15] (15) to (17);
		\draw [style=Arc, draw={black!30!green}, bend left=15] (15) to (17);
		\draw [style=Arc, draw=orange, bend left=345] (17) to (16);
		\draw [style=Arc, draw={black!30!green}, bend right=15] (16) to (17);
		\draw [style=Arc, draw={black!30!green}, bend right=15] (20) to (21);
		\draw [style=Arc, draw={black!30!green}, bend left=15] (20) to (21);
		\draw [style=Arc, draw=orange, bend right=15] (22) to (20);
		\draw [style=Arc, draw={black!30!green}, bend right=15] (20) to (22);
		\draw [style=Arc, draw={black!30!green}, bend right=15] (21) to (23);
		\draw [style=Arc, draw={black!30!green}, bend left=15] (21) to (23);
		\draw [style=Arc, draw=orange, bend left=345] (23) to (22);
		\draw [style=Arc, draw={black!30!green}, bend right=15] (22) to (23);
		\draw [style=Arc, draw=orange, bend left=330, looseness=1.25] (21) to (20);
		\draw [style=Arc, draw=orange, bend left=300, looseness=1.25] (21) to (20);
		\draw [style=Arc, draw=orange, bend right=45] (23) to (21);
		\draw [style=Arc, draw=orange, bend right=75, looseness=1.25] (23) to (21);
		\draw [bend left=90, looseness=1.75] (25.center) to (26.center);
		\draw [bend right=90, looseness=1.75] (25.center) to (26.center);
		\draw (41.center) to (40.center);
		\draw [style=Arc] (43.center) to (44.center);
		\draw (47.center) to (48.center);
		\draw (50.center) to (49.center);
		\draw (51.center) to (52.center);
		\draw (53.center) to (54.center);
		\draw [style=Arc, draw=orange] (62.center) to (17);
		\draw [style=Arc, draw=orange] (58.center) to (17);
		\draw [style=Arc, draw=orange] (15) to (56.center);
		\draw [style=Arc, draw=orange] (15) to (60.center);
		\draw [style=Arc, draw=orange] (59.center) to (14);
		\draw [style=Arc, draw=orange] (57.center) to (16);
		\draw [style=Arc, draw={black!30!green}] (16) to (61.center);
		\draw [style=Arc, draw={black!30!green}] (14) to (55.center);
		\draw (1) to (63.center);
		\draw (0) to (64.center);
		\draw (2) to (65.center);
		\draw (3) to (66.center);
		\draw [style=Arc, draw=orange] (71.center) to (23);
		\draw [style=Arc, draw=orange] (72.center) to (23);
		\draw [style=Arc, draw=orange] (69.center) to (22);
		\draw [style=Arc, draw=orange] (67.center) to (20);
		\draw [style=Arc, draw=orange] (21) to (73.center);
		\draw [style=Arc, draw=orange] (21) to (74.center);
		\draw [style=Arc, draw={black!30!green}] (20) to (68.center);
		\draw [style=Arc, draw={black!30!green}] (22) to (70.center);
		\draw [style=Arc, draw={black!30!green}] (23) to (78.center);
		\draw [style=Arc, draw={black!30!green}] (23) to (77.center);
		\draw [style=Arc, draw={black!30!green}] (75.center) to (21);
		\draw [style=Arc, draw={black!30!green}] (76.center) to (21);
	\end{pgfonlayer}
\end{tikzpicture}}
	\caption{Illustration of the proof, focussing on a 4-cycle}
	\label{fig:proofSetOfEvenCycles}
\end{figure}
		
		We now assume, for the sake of contradiction, that $G$ admits a circular $r$-coloring for some $r< \frac{2k}{k-1}$. This induces a circular $r$-coloring on the underlying graph of $G_1$ (and $G_2$). Consider the associated green-orange coloring based on the clockwise extension. See \Cref{fig:proofSetOfEvenCycles} for a pictorial description. The following three observations then lead us to a contradiction: 
	\begin{itemize}[itemsep=0pt, topsep=0pt]
\item[1.] By \Cref{prop:2k-cycle}, each cycle in $\mathcal{C}$ has half of its edges colored green, implying that the total number of green edges of $G_1$ is half the number of its edges. 
\item[2.] Adding two $x-y$ edges both directed the same direction does not change the parity of the number of green edges, implying that the parity of the number of green edges in $G_2$ remains as that of $G_1$. 
\item[3.] As we can pair up edges in $G_2$ into disjoint union of $\{x\to y, y\to x\}$, the total number of green edges of $G_2$ is also half the number of edges in $G_2$. But by our assumption, the difference of the half of the number of edges of $G_1$ and $G_2$ is an odd number.
	\end{itemize}
	\end{proof}
		
	In the proof above, a contradiction was reached based on the parity of the number of green edges of $G_1$ and $G_2$. The only place where the fact that $\phi$ is a circular coloring was used is in the first of the three statements, where we apply \Cref{prop:2k-cycle}. This proposition, assuming $\phi$ is a circular $r$-coloring with $r< \frac{2k}{k-1}$, gives a precise value of $k$ for the winding number of each cycle in $\mathcal{C}$. However, all that is needed for the proof is that the winding number of each of them has the same parity as $k$. This can be restated as follows.
	
	\begin{theorem}\label{thm:setOfEvenCyclesStronger}
		Let $\mathcal{C}$ be a set of even cycles of a graph $G$ each of length at most $2k$ where the degree of each edge in $\cal C$ is even. Furthermore, suppose for an assignment $D$ of directions to the cycles in $\mathcal{C}$, the sum of the  $(\mathcal{C}, D)$-degrees of the edges is $2 \pmod 4$, and that the subgraph induced by $\mathcal{C}$ is not bipartite. Then for any mapping $\phi$ of the vertices of $G$ to $\cal O$ that never identically map adjacent vertices, at least for one cycle in $\mathcal{C}$, we have: $w\curve{\vec{C}, \phi, \clock} \cong k+1 \pmod 2$.
	\end{theorem}

	The condition on $\phi$, that it does not identify adjacent vertices, can be dropped by assigning an order to vertices that are mapped to the same point: If $x$ is ordered before $y$, then extending $xy$ in clockwise direction means all point of $xy$ is mapped to same point and thus $xy$ is an orange edge. If $x$ is ordered after $y$, then the edge $xy$ maps fully around the circle and hence is a green edge.
	
	In proving that nonbipartite quadrangulations of the projective plane are 4-chromatic, Youngs \cite{Y96} proved a stronger result that in any 4-coloring of such a graph there is a face which receives all four colors. This claim is strengthened in \cite{HRS02} to any $k$-coloring rather than 4-coloring. Then in \cite{M02} it is shown that in any $k$-coloring of a nonbipartite quadrangulation of the projective plane there are at least three faces each receiving four different colors. We show here that the previous theorem implies all these claims. The first and second are immediate: Viewing the $k$ colors as the points of the circle $\mathcal{O}$ with cyclic distance of at least 1 between consecutive points, the 4-cycle with the winding number 1 or 3 must have received 4-different colors. We note that for the claim of having at least one face with winding number 1 or 3, one such face is the best we can expect. A projective planar embedding of $M(C_5)$ where all faces are 4-cycles, together with a 4-coloring and orientations of the faces is given in \Cref{fig:M2C5}. In this setting there is only one face with odd winding number. However, the stronger claim of having three 4-cycles that are colored with four colors follows from reordering the colors on $\mathcal{O}$. For example, if we change the cyclic ordering of the colors from $1234$ to $2134$, then the 4-cycle which had the winding number 1 in the first ordering, now is of winding number 2. Thus another 4-cycle must have an odd winding number. The third cycle is obtained by considering the cyclic order $1243$ of colors. Up to cyclic ordering of the colors on the circle, and after identifying clockwise and counterclockwise orderings of the colors, three 4-colored 4-faces is the best we could expect when we are working with quadrangulations of the projective plane. One can obtain similar conclusions when larger faces are considered.

	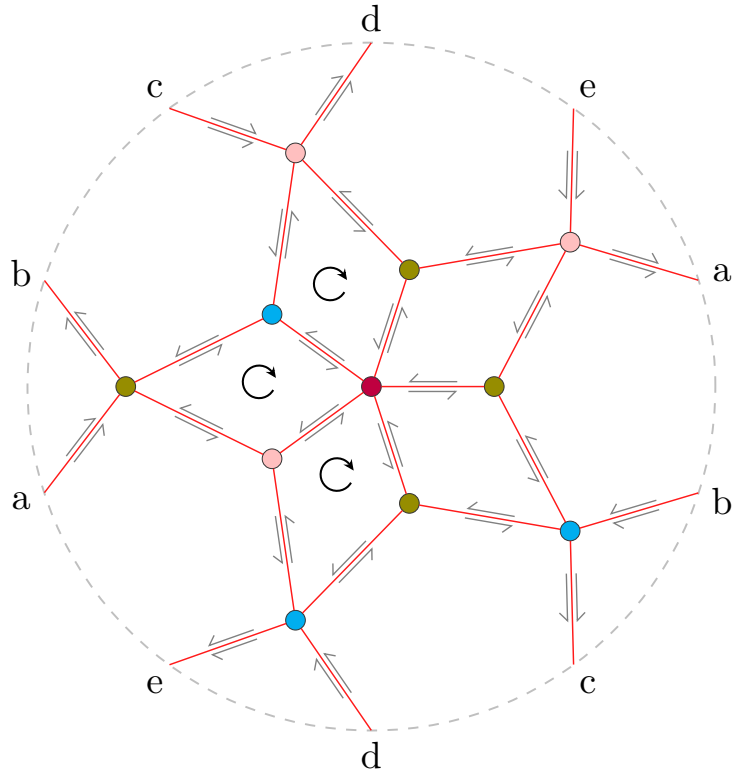
\begin{figure}[ht]
		\centering
		\scalebox{1}{	\begin{minipage}[t]{1\textwidth}
		\centering
		\scalebox{1.3}{
			\begin{tikzpicture}
				
				\foreach \i in {1,4,5}
				{ \draw node[line width=0.05mm, circle, fill=olive, minimum width=.5mm, inner sep=0mm, draw=black!80, minimum size=2mm]  (\i) at ({\i*360/5}:2.5){};}

					\foreach \i in {2}
				{ \draw node[line width=0.05mm, circle, fill=cyan, minimum width=.5mm, inner sep=0mm, draw=black!80, minimum size=2mm]  (\i) at ({\i*360/5}:2.5){};}
				
					\foreach \i in {3}
				{ \draw node[line width=0.05mm, circle, fill=pink, minimum width=.5mm, inner sep=0mm, draw=black!80, minimum size=2mm]  (\i) at ({\i*360/5}:2.5){};}

				\foreach \i in {6,10}
				{ \draw node[line width=0.05mm, circle, fill=pink, minimum width=.5mm, inner sep=0mm, draw=black!80, minimum size=2mm]  (\i) at ({\i*360/5+180/5}:5){};
				}
				
					\foreach \i in {8,9}
				{ \draw node[line width=0.05mm, circle, fill=cyan, minimum width=.5mm, inner sep=0mm, draw=black!80, minimum size=2mm]  (\i) at ({\i*360/5+180/5}:5){};
				}
				
					\foreach \i in {7}
				{ \draw node[line width=0.05mm, circle, fill=olive, minimum width=.5mm, inner sep=0mm, draw=black!80, minimum size=2mm]  (\i) at ({\i*360/5+180/5}:5){};
				}
				
			   	\foreach \i in {11, 12, ..., 20}
			   { \draw node[line width=0.05mm, circle, fill=white, minimum width=.5mm, inner sep=0mm, draw=black!80, minimum size=0mm]  (\i) at ({\i*360/10+180/10}:7){};
			   }
				\draw ({11*360/10+180/10}:7.5) node {e};
				\draw ({12*360/10+180/10}:7.5) node {d};
				\draw ({13*360/10+180/10}:7.5) node {c};
				\draw ({14*360/10+180/10}:7.5) node {b};
				\draw ({15*360/10+180/10}:7.5) node {a};
				\draw ({16*360/10+180/10}:7.5) node {e};
				\draw ({17*360/10+180/10}:7.5) node {d};
				\draw ({18*360/10+180/10}:7.5) node {c};
				\draw ({19*360/10+180/10}:7.5) node {b};
				\draw ({20*360/10+180/10}:7.5) node {a};
				
				\draw [ thick, red, line width=0.15mm, opacity=0.9] (1) -- (6);
				\draw [ thick, red, line width=0.15mm, opacity=0.9] (2) -- (7);
				\draw [ thick, red, line width=0.15mm, opacity=0.9] (3) -- (8);
				\draw [ thick, red, line width=0.15mm, opacity=0.9] (4) -- (9);
				\draw [ thick, red, line width=0.15mm, opacity=0.9] (5) -- (10);
				
				\draw [ thick, red, line width=0.15mm, opacity=0.9] (1) -- (10);
				\draw [ thick, red, line width=0.15mm, opacity=0.9] (2) -- (6);
				\draw [ thick, red, line width=0.15mm, opacity=0.9] (3) -- (7);
				\draw [ thick, red, line width=0.15mm, opacity=0.9] (4) -- (8);
				\draw [ thick, red, line width=0.15mm, opacity=0.9] (5) -- (9);

				\draw  [ dashed, gray, line width=0.2mm, opacity=0.5] (0,0) circle (7cm);
				
				\draw node[line width=0.05mm, circle, fill=purple, minimum width=.5mm, inner sep=0mm, draw=black!80, minimum size=2mm]  (90) at ({90}:0){};
				
				\draw [ thick, red, line width=0.15mm, opacity=0.9] (1) -- (90);
				\draw [ thick, red, line width=0.15mm, opacity=0.9] (2) -- (90);
				\draw [ thick, red, line width=0.15mm, opacity=0.9] (3) -- (90);
				\draw [ thick, red, line width=0.15mm, opacity=0.9] (4) -- (90);
				\draw [ thick, red, line width=0.15mm, opacity=0.9] (5) -- (90);

				\draw [ thick, red, line width=0.15mm, opacity=0.9] (6) -- (12);
				\draw [ thick, red, line width=0.15mm, opacity=0.9] (6) -- (13);
				
				\draw [ thick, red, line width=0.15mm, opacity=0.9] (7) -- (14);
				\draw [ thick, red, line width=0.15mm, opacity=0.9] (7) -- (15);
				
				\draw [ thick, red, line width=0.15mm, opacity=0.9] (8) -- (16);
				\draw [ thick, red, line width=0.15mm, opacity=0.9] (8) -- (17);
				
				\draw [ thick, red, line width=0.15mm, opacity=0.9] (9) -- (18);
				\draw [ thick, red, line width=0.15mm, opacity=0.9] (9) -- (19);
				
				\draw [ thick, red, line width=0.15mm, opacity=0.9] (10) -- (20);
				\draw [ thick, red, line width=0.15mm, opacity=0.9] (10) -- (11);

					\foreach \i/\j in {90/1,90/2,90/3,90/4,90/5,1/90,2/90,3/90,4/90,5/90}
				{ 	\draw[-{Straight Barb[right]}, gray]      (\i) to[parallel segment]      (\j);}

				\foreach \i/\j in {1/6,6/2, 2/7,7/3,3/8,8/4,4/9,9/5,5/10,10/1}
				{ 	\draw[-{Straight Barb[right]}, gray]      (\i) to[parallel segment]      (\j);}

				\foreach \i/\j in {6/1,2/6, 7/2,3/7,8/3,4/8,9/4,5/9,10/5,1/10}
				{ 	\draw[-{Straight Barb[right]}, gray]      (\i) to[parallel segment]      (\j);}

				\foreach \i/\j in {13/6,7/14, 15/7,8/16,17/8,9/18,19/9,10/20,11/10,6/12}
				{ 	\draw[-{Straight Barb[right]}, gray]      (\i) to[parallel segment]      (\j);}
				
					\foreach \i/\j in {13/6,7/14, 15/7,8/16,17/8,9/18,19/9,10/20,11/10,6/12}
				{ 	\draw[-{Straight Barb[left]}, gray]      (\i) to[-parallel segment]      (\j);}

 \foreach \i in {1}
 { \draw[line width=0.6pt,
 	-{Stealth[length=1.1mm, width=1mm]}]  ({\i*360/5+170/5}:2) arc[start angle=-30, end angle=-335, radius=2ex]
 		-- ++(305:0.45ex);
 }
 
  \foreach \i in {2}
 { \draw[line width=0.6pt,
 	-{Stealth[length=1.1mm, width=1mm]}]  ({\i*360/5+190/5}:2) arc[start angle=-30, end angle=-335, radius=2ex]
 	-- ++(305:0.45ex);
 }
 
  \foreach \i in {3}
 { \draw[line width=0.6pt,
 	-{Stealth[length=1.1mm, width=1mm]}]  ({\i*360/5+210/5}:2) arc[start angle=-30, end angle=-335, radius=2ex]
 	-- ++(305:0.45ex);
 }
				
			\end{tikzpicture}
		}

	\end{minipage}}        
		\caption{For each positioning of the four colors on $\mathcal{O}$ one of the three indicated faces will have winding number $+1$ or $-1$.}
		\label{fig:M2C5}
	\end{figure}
	
   The proof above could also be presented in terms of the difference between the number of green and orange edges: If $\epsilon$ is small enough, then in a $(2+\epsilon)$-coloring of a $2k$-cycle, with respect to the clockwise extension, the number of green and orange edges are the same. Clockwise extension of a (directed) 2-cycle always has one green and one orange edge. But if a pair of $xy$-edges (transitive 2-cycle) are extended to the same arc of $\mathcal{O}$, then the pair contributes $2$ to one of the green or orange family and 0 to the other. Thus, on the one hand the difference of the number of green and orange edges is $0 \mod 4$ and on the other hand it is, in  $\pmod 4$, the same as the number of edges remaining after removing every possible directed 2-cycle. Thus if the latter is a $2 \mod 4$, we have contradiction which only means $G$ admits no such coloring.

	\subsection{Odd cycles}
	
 If a planar triangulation contains a vertex of odd degree, then it contains an odd wheel. Hence its chromatic number and also its circular chromatic number, is at least 4 and by the four-color theorem is precisely 4. Thus for a planar triangulation to admit a 3-coloring or even a circular $(4-\epsilon)$-coloring ($\epsilon > 0$) it must be Eulerian. It is a classic result with a number of interesting proofs in the literature that this necessary condition is also sufficient. This result is extended to triangulations of other surfaces in \cite{DGMVZ05} with an implicit use of the notion of signed dual. Here applying \Cref{prop:2k+1-cycle-II} together with the notion of signed dual, we provide a forbidden interval for the value of the circular chromatic number if certain cycles system is found inside a graph.
 
 \begin{theorem}\label{thm:balanced-dual}
 	Let $G$ be a graph and $\mathcal{C}$ be a set of $(2k+1)$-cycles of $G$. If $\widehat S(\cal C)$ is not balanced, then $\chi_c(G)\geq \frac{4k}{2k-1}$. 
 \end{theorem}    
	
\begin{proof}
	Suppose $\chi_c(G) < \frac{4k}{2k-1}$. By \Cref{prop:2k+1-cycle-II}, for each cycle $C$ in $\cal C$, one of the directions say $D_C$ gives winding number $k$. By \Cref{cor:OddCyclesCommonEdge}, under the directions $D = \set{D_C, C \in \cal C}$, the dual signed graph $S(\cal C,D)$ has no negative edge. Finally, for any other choice $D'$ of the directions on $\cal C$ the dual signed graph $S(\cal C,D')$ is switching equivalent to $S(\cal C,D)$ (by \Cref{obs:switch_dual}) and, therefore, is balanced.
\end{proof}

This generalizes the fact that a planar triangulations which is not Eulerian, has circular chromatic number (at least) 4. As noted, it is a well known fact that the condition of being Eulerian is also sufficient in this case for the graph to map to $K_3$. In the next section we use the above theorem to extend such results also to Eulerian graphs on surfaces such as plane and projective plane where each face is a $(2k+1)$-cycle.  

	\subsection{Graphs on surfaces}
 Here we see natural applications of the previous results on graph embedded on certain surfaces, most notably on the projective plane. For standard facts and terminologies we refer to \cite{Mohar-Thomassen-Book}. Let us, however, remind a reader that if all faces of a graph embedded on the projective plane are even, then all noncontractible cycles are same parity, but that both are possible. Examples of two cases are given in \Cref{fig:Odd-EvenNonContractible}. To apply the results of the previous section to graphs on surface, one may first select a set of noncontractible cycles in such a way that cutting the surface along them we get a surface homeomorphic to the plane, orienting them in the direction they are cut. We then take all facial cycles of the resulting plane graph each oriented in the clockwise direction. In what follows we first show how our results can be applied to embedded graphs where all faces are even, thus providing somewhat simpler proof for result in the literature (see \cite{DGMVZ05} and references therein). We then obtain similar gaps for embedded graphs where all faces are odd-cycles of a given length. This extend the classic result about 3-colorability of triangulations of the plane.    
	
	\begin{figure}
		\centering
		\centering
\begin{tikzpicture}[scale=.5]
\coordinate (u1) at (-0.1200,3.9900);
\coordinate (u2) at (-0.1400,0.6500);
\coordinate (u21) at (-3.1000,-0.9700);
\coordinate (u22) at (-3.1200,-2.5606);
\coordinate (u23) at (-3.1200,2.3700);
\coordinate (u24) at (-3.1200,5.7700);
\coordinate (u25) at (-3.1200,7.3006);
\coordinate (u37) at (-5.4645,9.9549);
\coordinate (u39) at (-6.0800,4.0500);
\coordinate (u41) at (-6.1800,0.8100);
\coordinate (u48) at (-7.4147,4.7921);
\coordinate (u49) at (-7.4600,0.0300);
\coordinate (u65) at (1.1796,4.7834);
\coordinate (u66) at (1.2470,0.0807);

\tikzset{graph edge/.style={line width=1.5pt,shorten >=0.8pt,shorten <=0.8pt}}

\draw [dashed, gray, line width=0.2mm, opacity=0.5] (u23) circle (4.9306cm);
\draw ($(u23)+({0*360/6+180/6}:5.3cm)$) node {a};
\draw ($(u23)+({1*360/6+180/6}:5.3cm)$) node {c};
\draw ($(u23)+({2*360/6+180/6}:5.3cm)$) node {b};
\draw ($(u23)+({3*360/6+180/6}:5.3cm)$) node {a};
\draw ($(u23)+({4*360/6+180/6}:5.3cm)$) node {c};
\draw ($(u23)+({5*360/6+180/6}:5.3cm)$) node {b};
\draw [line width=1pt] (u39)-- (u23);
\draw [line width=1pt] (u23)-- (u21);
\draw [line width=1pt] (u21)-- (u41);
\draw [line width=1pt] (u41)-- (u39);
\draw [line width=1pt] (u1)-- (u23);
\draw [line width=1pt] (u24)-- (u1);
\draw [line width=1pt] (u39)-- (u24);
\draw [line width=1pt] (u2)-- (u1);
\draw [line width=1pt] (u21)-- (u2);
\draw [line width=1pt] (u41)-- (u49);
\draw [line width=1pt] (u48)-- (u39);
\draw [line width=1pt] (u24)-- (u25);
\draw [line width=1pt] (u1)-- (u65);
\draw [line width=1pt] (u2)-- (u66);
\draw [line width=1pt] (u21)-- (u22);

\begin{scriptsize}
\draw [fill=white] (u39) circle (5.1pt);
\draw [fill=white] (u23) circle (5.1pt);
\draw [fill=white] (u41) circle (5.1pt);
\draw [fill=white] (u21) circle (5.1pt);
\draw [fill=white] (u1) circle (5.1pt);
\draw [fill=white] (u24) circle (5.1pt);
\draw [fill=white] (u2) circle (5.1pt);

\end{scriptsize}
\end{tikzpicture}
		\qquad
		\begin{tikzpicture}[line cap=round,line join=round,>=triangle 45,x=0.9068cm,y=0.9068cm,scale=.5,rotate=180]
\coordinate (u68) at (10.6613,5.6931);
\coordinate (u70) at (10.6813,2.1928);
\coordinate (u72) at (10.7049,-1.9526);
\coordinate (u75) at (11.4301,-3.1923);
\coordinate (u89) at (13.7226,0.4600);
\coordinate (u93) at (14.2595,4.2860);
\coordinate (u96) at (14.9705,5.5339);
\coordinate (u101) at (15.6957,4.2942);
\coordinate (u110) at (5.6432,4.2369);
\coordinate (u112) at (6.3542,5.4848);
\coordinate (u116) at (7.0794,4.2451);
\coordinate (u119) at (7.6599,0.4254);
\coordinate (u122) at (8.1832,10.1665);
\coordinate (u133) at (9.9939,-3.2005);

\tikzset{graph edge/.style={line width=1.5pt,shorten >=0.8pt,shorten <=0.8pt}}

\draw [dashed, gray, line width=0.2mm, opacity=0.5] (u70) circle (4.9306cm);

\draw ($(u70)+({-0.17*360/6+180/6}:5.3cm)$) node {a};%
\draw ($(u70)+({0.17*360/6+180/6}:5.3cm)$) node {c};
\draw ($(u70)+({1.83*360/6+180/6}:5.3cm)$) node {b};%
\draw ($(u70)+({2.17*360/6+180/6}:5.3cm)$) node {a};
\draw ($(u70)+({3.83*360/6+180/6}:5.3cm)$) node {c};%
\draw ($(u70)+({4.17*360/6+180/6}:5.3cm)$) node {b};

\draw [line width=1pt] (u70)-- (u68);
\draw [line width=1pt] (u70)-- (u119);
\draw [line width=1pt] (u70)-- (u89);
\draw [line width=1pt] (u89)-- (u93);
\draw [line width=1pt] (u68)-- (u93);
\draw [line width=1pt] (u68)-- (u116);
\draw [line width=1pt] (u116)-- (u119);
\draw [line width=1pt] (u119)-- (u72);
\draw [line width=1pt] (u72)-- (u89);
\draw [line width=1pt] (u116)-- (u112);
\draw [line width=1pt] (u116)-- (u110);
\draw [line width=1pt] (u133)-- (u72);
\draw [line width=1pt] (u72)-- (u75);
\draw [line width=1pt] (u101)-- (u93);
\draw [line width=1pt] (u93)-- (u96);

\begin{scriptsize}

\draw [fill=white] (u70) circle (5.1pt);
\draw [fill=white] (u68) circle (5.1pt);
\draw [fill=white] (u119) circle (5.1pt);
\draw [fill=white] (u89) circle (5.1pt);
\draw [fill=white] (u93) circle (5.1pt);
\draw [fill=white] (u116) circle (5.1pt);
\draw [fill=white] (u72) circle (5.1pt);

\end{scriptsize}
\end{tikzpicture}
		\caption{Quadrangulations of the projective plane with even (left) and odd (right) noncontractible cycles.}
		\label{fig:Odd-EvenNonContractible}
	\end{figure}
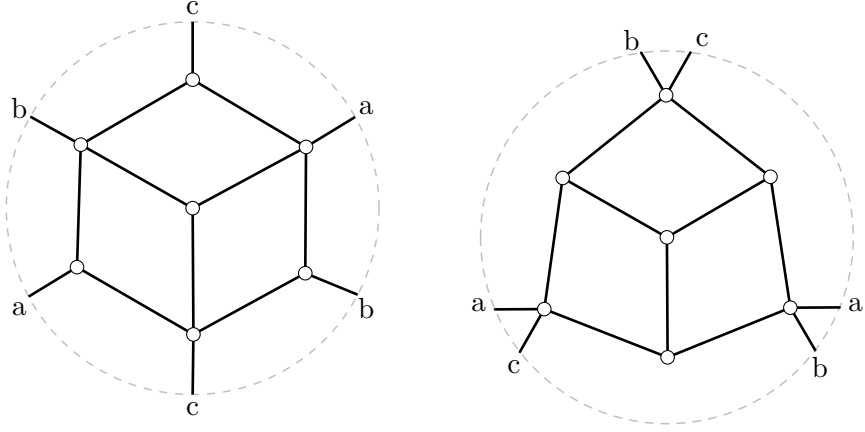
	
		\begin{theorem}
		If $G$ is a graph embedded on the projective plane where each face is an even cycle of length at most $2k$  ($k\geq 2$), then either $G$ is bipartite or satisfies $\chi_c(G)\geq \frac{2k}{k-1} $.
	\end{theorem}
	
\begin{proof}
		We assume the graph is not bipartite, so there exists an odd cycle $C$, which is not contractible because all faces are even.
	Let $G/C$ be the plane graph obtained from $G$ by cutting along $C$. Observe that in $G/C$ each edge of $C$ appears twice on the outer face. Let $\cal C$ be the set of cycles of $G$, corresponding to all bounded faces of $G/C$. The assignment of directions $D$ is the standard one, that is to assign clockwise direction to all faces based on the planar embedding of $G/C$.
	Observe that all the internal edges (meaning the edges not in $C$) are used in exactly two faces in two opposite orientations therefore, their $(\cal C,D)$-degree is $0$. In contrast, the edges of $C$ will all be oriented twice in the same direction so their $(\cal C,D)$-degree is $2$. Since there is an odd number of edges in $C$, we have that the sum of $(\cal C,D)$-degrees is $2\pmod 4$, so \Cref{thm:setOfEvenCycles} applies and we obtain the desired result.
	\end{proof}

	For higher surfaces, we have the following, which is written as an extension of \cite{MS02}. In particular, given a graph $G$ embedded on a surface $S$, when cutting along a cycle $C$ we consider both the geometric notion and graph notion.
	
	\begin{theorem}
		Suppose $G$ is a graph with a 2-cell embedding on a surface where each face is an even cycle of length at most $2k$ ($k\geq 2$). Assume there are disjoint cycles $C_1,\dots,C_g$ such that after cutting along all of them, we obtain a sphere with $g$ holes and $g$ Möbius strips, an odd number of which are non-bipartite, then $\chi_c(G) \geq \frac{2k}{k-1}$.
	\end{theorem}
	\begin{proof}
		Consider the set $\cal C$ of the facial cycles of $G$. For the Möbius strip obtain after cutting along $C_i$, consider a noncontractible cycle $C_i'$ of it (going once around the strip). Note that the Möbius strip is bipartite if and only if $C_i'$ is of even length. Observe that each $C_i'$ is a non-separating cycle of the surface. Furthermore, after cutting the original surface along $C_1',\dots,C_g'$, we obtain a sphere with $g$ holes, denoted $\Gamma$. Then every face of $\Gamma$ can be oriented clockwise. This gives an orientation to each cycle of $\cal C$, resulting in the system $(C,D)$.
		
		For any edge $e$, if $e$ is not in any of the $C_i'$ then $e$ is in two adjacent faces in $\Gamma$ and its $(\cal C,D)$-degree is $0$. Otherwise, $e$ appear twice in $\Gamma$, bordering the hole made by some $C_i'$ in which case, its $(\cal C,D)$-degree is $2$.
		Since an odd number of the $C_i'$ are non-bipartite, the sum of the $(\cal C,D)$-degrees is $2 \pmod 4$ and we may apply \Cref{thm:setOfEvenCyclesStronger} for a contradiction. \end{proof}
	
	A face walk in a graph $G$ embedded on a surface is a walk in the signed dual (using the faces as the set of cycles). It is noncontractible if the walk consists of a noncontractible cycle.
	
	\begin{theorem}\label{thm:proj_planar_odd}
		If $G$ is a graph embedded on the projective plane, whose faces are $2k+1$-cycles, then: 
		\begin{itemize}		
			\item If $G$ has an odd degree vertex or an even length noncontractible face-walk, then $\chi_c(G) \geq \frac{4k}{2k-1}$
			\item Otherwise $G$ is eulerian with all noncontractible face-walks of odd size and $\chi_c(G) = \frac{2k+1}{k}$.
		\end{itemize}
	\end{theorem}

	\begin{proof}
	
		Consider the set $\cal C$ of the faces of $G$. If the dual $\widehat S(\cal C)$ is unbalanced, then by \Cref{thm:balanced-dual}, $\chi_c(G) \geq \frac{4k}{2k-1}$.
		If the dual is balanced, then we show that $\chi_c(G) = \frac{2k+1}{k}$ ($G$ maps to its shortest cycle).
		Since the dual is balanced, there exists a switching of the vertices such that all edges are positive. Equivalently, we may chose a direction for each face such that each edge is oriented in the same direction in both faces containing it. This induces an orientation $\tau$ of the edges of $G$, that we will show to be a $(\mathbb{Z}_{2k+1}, +1)$-tension. 
	For that, we need to show that for any closed walk of $G$, the weighted sum of the edges (with respect to $\tau$) is $0 \mod 2k+1$.
	Actually, it is enough to prove it on a generating set of cycles: All the faces and one noncontractible cycle, say $C$. Since the faces are oriented as directed cycles, and are of length $2k+1$, the property holds for them.
	Furthermore, suppose the noncontractible cycle $C$ sums to $x \mod 2k+1$, then the walk $2C$ going twice along $C$ sums to $2x \mod 2k+1$. But, by considering cutting along $C$, we observe that $2C$ is a binary sum of faces and thus $2x = 0 \mod 2k+1$ so $x = 0 \mod 2k+1$.
	Therefore, $\tau$ is indeed a $(\mathbb{Z}_{2k+1}, \pm 1)$-tension, so $G$ admits a homomorphism to $C_{2k+1}$, in other words, $\chi_c(G) = \frac{2k+1}{k}$.
	
	Finally, it remains to show that $\widehat S(\cal C)$ is balanced if and only if $G$ is eulerian and has no even noncontractible face-walk.
	Since reversing the direction of a face is equivalent to switching the corresponding vertex in $\widehat S(\cal C)$, we may work with any direction on the faces. To show that it is eulerian, consider the face-walk $W_x$ around a vertex $x$ of $G$. Noting the local planarity around the vertex and orienting the faces clockwise, we conclude that all edges of $W_x$ are negative therefore, in order to be balanced, it needs to be of even length.
	Moreover, let $W$ be a noncontractible walk. We view $G$ as embedded on the plane, with a cross-cap. Observe that an odd number of faces in $W$ must intersect the cross-cap. We direct the faces of $G$ which do not intersect the cross-cap in the clockwise direction, and the other ones are directed arbitrarily.
	Each face of $W$ intersecting the cross-cap introduces one positive and one negative edge, the other ones, each introduce two negative edges.
	Thus, the walk corresponding to $W$ in $\widehat S(\cal C)$ has an odd number of positive edges, hence, it is balanced if and only if it is of odd length.
	\end{proof}
	
	An example of a triangle-free eulerian graph embedded on the projective plane such that all faces are 5-cycles is given \Cref{fig:EulerainFiveFaces}. Since the dual is not balanced, its circular chromatic number is at least 3 and in this example. since it is 3-degenerate, its circular chromatic number is exactly 3.
	
	\begin{figure}
		\centering
		\scalebox{0.8}{\input{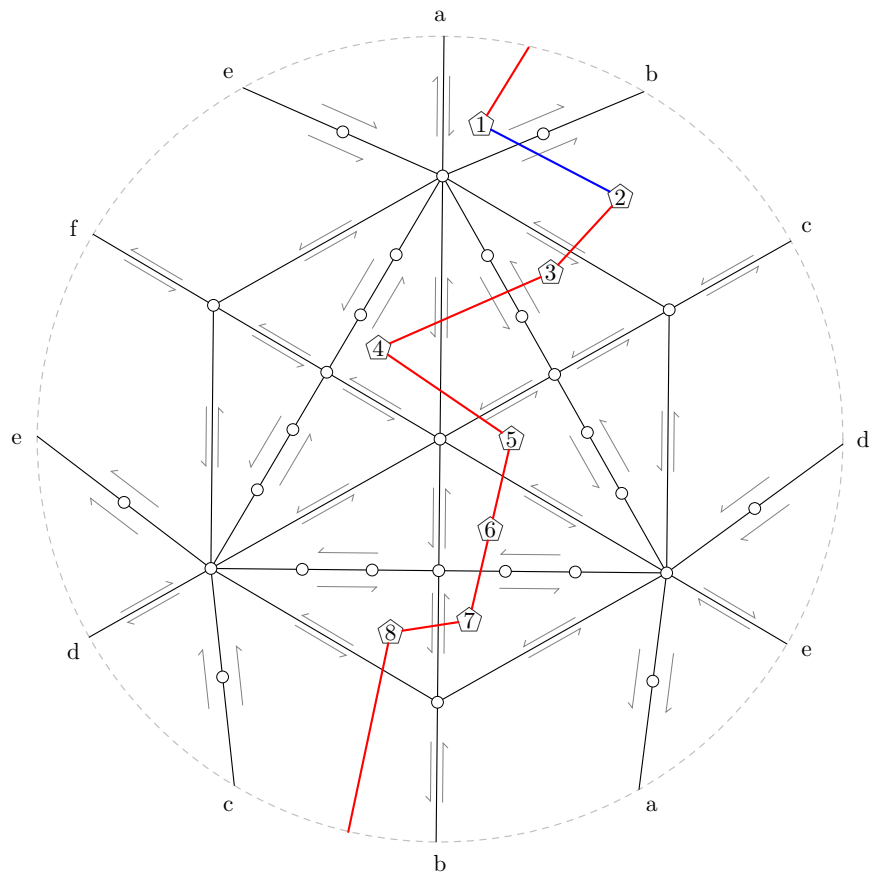}}
		\caption{An eulerian projective planar graph, with a negative face walk }
		\label{fig:EulerainFiveFaces}
	\end{figure}
	
	\begin{corollary}\label{coro:planarOddFace}
		Let $G$ be a graph embedded on the plane such that all faces are $2k+1$-cycles. If $G$ has an odd degree vertex, then $\chi_c(G) \geq \frac{4k}{2k-1}$.
	\end{corollary}
	
	 \begin{corollary}\label{coro:Odd-face-dichotomy}
		If $G$ is projective planar graph with an embedding where all faces are 5-cycles, 
		then either $G$ is eulerian and all noncontractible facial walks are of odd length in which case $\chi_c(G)=\frac{5}{2}$ otherwise $\chi_c(G)\geq 3$.            	
	\end{corollary}
	
	For higher surfaces, \Cref{thm:balanced-dual} can still be applied, therefore, if the dual $\widehat S(\cal C)$ is unbalanced, then $\chi_c(G) \geq \frac{4k}{2k-1}$. However, the argument that, if $\widehat S(\cal C)$ is balanced, then $G$ admits a homomorphism to $C_{2k+1}$ doesn't hold anymore. \Cref{fig:KleinBottleCounterExample} shows an example of a triangulation of a Klein bottle which is not 3-colorable despite having a balanced signed dual.

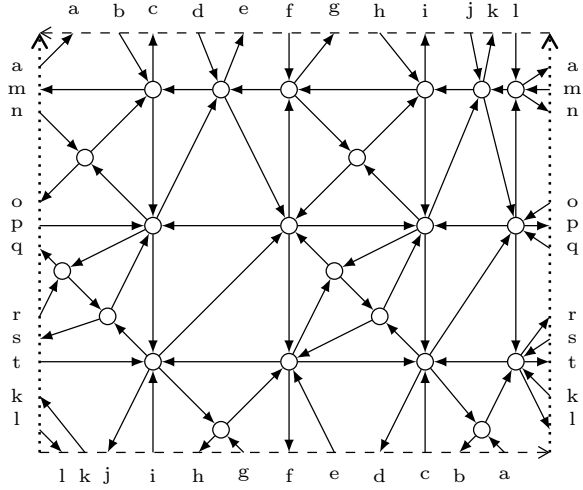
\begin{figure}[!htb]
	\centering
	\scalebox{1.2}{\begin{tikzpicture}
	\begin{pgfonlayer}{nodelayer}
		\node [style=small black empty] (7) at (-3, -3) {};
		\node [style=small black empty] (8) at (-4, -2) {};
		\node [style=small black empty] (12) at (-4.5, 1.5) {};
		\node [style=small black empty] (14) at (-3, 3) {};
		\node [style=small black empty] (15) at (-3, 0) {};
		\node [style=small black empty] (16) at (-5, -1) {};
		\node [style=small black empty] (17) at (-1.5, 3) {};
		\node [style=small black empty] (18) at (-1.5, -4.5) {};
		\node [style=none] (20) at (-4.75, 4.25) {};
		\node [style=none] (21) at (-3.75, 4.25) {};
		\node [style=none] (22) at (-3, 4.25) {};
		\node [style=none] (23) at (-2, 4.25) {};
		\node [style=none] (24) at (-1, 4.25) {};
		\node [style=none] (27) at (-1, -5) {};
		\node [style=none] (28) at (-2, -5) {};
		\node [style=none] (29) at (-3, -5) {};
		\node [style=none] (30) at (-4, -5) {};
		\node [style=none] (57) at (0, -5) {};
		\node [style=small black empty] (62) at (0, -3) {};
		\node [style=small black empty] (63) at (0, 0) {};
		\node [style=small black empty] (64) at (0, 3) {};
		\node [style=none] (66) at (-5.5, 3.5) {};
		\node [style=none] (67) at (-5.5, 3) {};
		\node [style=none] (68) at (-5.5, 2.5) {};
		\node [style=none] (69) at (-5.5, 0.5) {};
		\node [style=none] (70) at (-5.5, 0) {};
		\node [style=none] (71) at (-5.5, -0.5) {};
		\node [style=none] (72) at (-5.5, -2) {};
		\node [style=none] (73) at (-5.5, -2.5) {};
		\node [style=none] (74) at (-5.5, -3) {};
		\node [style=none] (85) at (-4.5, -5) {};
		\node [style=none] (86) at (-5.5, -3.75) {};
		\node [style=none] (92) at (-5.5, 4.25) {};
		\node [style=none] (93) at (-5.5, -5) {};
		\node [style=small black empty] (95) at (3, -3) {};
		\node [style=small black empty] (96) at (2, -2) {};
		\node [style=small black empty] (97) at (1.5, 1.5) {};
		\node [style=small black empty] (98) at (3, 3) {};
		\node [style=small black empty] (99) at (3, 0) {};
		\node [style=small black empty] (100) at (1, -1) {};
		\node [style=small black empty] (101) at (4.25, 3) {};
		\node [style=small black empty] (102) at (4.25, -4.5) {};
		\node [style=none] (103) at (1, 4.25) {};
		\node [style=none] (104) at (2, 4.25) {};
		\node [style=none] (105) at (3, 4.25) {};
		\node [style=none] (106) at (4, 4.25) {};
		\node [style=none] (107) at (4.5, 4.25) {};
		\node [style=none] (108) at (4.75, -5) {};
		\node [style=none] (109) at (3.75, -5) {};
		\node [style=none] (110) at (3, -5) {};
		\node [style=none] (111) at (2, -5) {};
		\node [style=none] (112) at (5.75, -5) {};
		\node [style=small black empty] (113) at (5, -3) {};
		\node [style=small black empty] (114) at (5, 0) {};
		\node [style=small black empty] (115) at (5, 3) {};
		\node [style=none] (125) at (5.75, 3.5) {};
		\node [style=none] (126) at (5.75, 3) {};
		\node [style=none] (127) at (5.75, 2.5) {};
		\node [style=none] (128) at (5.75, 0.5) {};
		\node [style=none] (129) at (5.75, 0) {};
		\node [style=none] (130) at (5.75, -0.5) {};
		\node [style=none] (131) at (5.75, -2) {};
		\node [style=none] (132) at (5.75, -2.5) {};
		\node [style=none] (133) at (5.75, -3) {};
		\node [style=none] (134) at (1, -5) {};
		\node [style=none] (136) at (5.75, -3.75) {};
		\node [style=none] (137) at (0, 4.25) {};
		\node [style=none] (139) at (5.75, 4.25) {};
		\node [style=none] (140) at (5.75, 4.25) {};
		\node [style=none] (143) at (5.75, -5) {};
		\node [style=none] (144) at (5.75, -4.5) {};
		\node [style=none] (145) at (-5.5, -4.5) {};
		\node [style=none] (146) at (-5, -5) {};
		\node [style=none] (147) at (5, 4.25) {};
		\node [style=none] (149) at (-3.75, 4.75) {\tiny{b}};
		\node [style=none] (150) at (-3, 4.75) {\tiny{c}};
		\node [style=none] (151) at (-2, 4.75) {\tiny{d}};
		\node [style=none] (152) at (-1, 4.75) {\tiny{e}};
		\node [style=none] (153) at (0, 4.75) {\tiny{f}};
		\node [style=none] (154) at (1, 4.75) {\tiny{g}};
		\node [style=none] (155) at (2, 4.75) {\tiny{h}};
		\node [style=none] (156) at (3, 4.75) {\tiny{i}};
		\node [style=none] (191) at (5, 4.75) {\tiny{l}};
		\node [style=none] (157) at (4, 4.75) {\tiny{j}};
		\node [style=none] (158) at (4.5, 4.75) {\tiny{k}};
		\node [style=none] (159) at (6.25, 3.5) {\tiny{a}};
		\node [style=none] (160) at (6.25, 2.5) {\tiny{n}};
		\node [style=none] (161) at (-4.75, 4.75) {\tiny{a}};
		\node [style=none] (162) at (6.25, 3) {\tiny{m}};
		\node [style=none] (163) at (6.25, 0.5) {\tiny{o}};
		\node [style=none] (164) at (6.25, -0.5) {\tiny{q}};
		\node [style=none] (165) at (6.25, 0) {\tiny{p}};
		\node [style=none] (166) at (6.25, -2) {\tiny{r}};
		\node [style=none] (167) at (6.25, -3) {\tiny{t}};
		\node [style=none] (168) at (6.25, -2.5) {\tiny{s}};
		\node [style=none] (169) at (-6, 3.5) {\tiny{a}};
		\node [style=none] (170) at (-6, 2.5) {\tiny{n}};
		\node [style=none] (171) at (-6, 3) {\tiny{m}};
		\node [style=none] (172) at (-6, 0.5) {\tiny{o}};
		\node [style=none] (173) at (-6, -0.5) {\tiny{q}};
		\node [style=none] (174) at (-6, 0) {\tiny{p}};
		\node [style=none] (175) at (-6, -2) {\tiny{r}};
		\node [style=none] (176) at (-6, -3) {\tiny{t}};
		\node [style=none] (177) at (-6, -2.5) {\tiny{s}};
		\node [style=none] (178) at (3.75, -5.5) {\tiny{b}};
		\node [style=none] (179) at (3, -5.5) {\tiny{c}};
		\node [style=none] (180) at (2, -5.5) {\tiny{d}};
		\node [style=none] (181) at (1, -5.5) {\tiny{e}};
		\node [style=none] (182) at (0, -5.5) {\tiny{f}};
		\node [style=none] (183) at (-1, -5.5) {\tiny{g}};
		\node [style=none] (184) at (-2, -5.5) {\tiny{h}};
		\node [style=none] (185) at (-3, -5.5) {\tiny{i}};
		\node [style=none] (186) at (-4, -5.5) {\tiny{j}};
		\node [style=none] (187) at (-4.5, -5.5) {\tiny{k}};
		\node [style=none] (194) at (-5, -5.5) {\tiny{l}};
		\node [style=none] (188) at (4.75, -5.5) {\tiny{a}};
		\node [style=none] (189) at (-6, -3.75) {\tiny{k}};
		\node [style=none] (193) at (-6, -4.25) {\tiny{l}};
		\node [style=none] (190) at (6.25, -3.75) {\tiny{k}};
		\node [style=none] (192) at (6.25, -4.25) {\tiny{l}};

	\end{pgfonlayer}
	\begin{pgfonlayer}{edgelayer}
		\draw [style=Arc] (7) to (8);
		\draw [style=Arc] (12) to (14);
		\draw [style=Arc] (14) to (15);
		\draw [style=Arc] (15) to (12);
		\draw [style=Arc] (15) to (16);
		\draw [style=Arc] (16) to (8);
		\draw [style=Arc] (8) to (15);
		\draw [style=Arc] (15) to (17);
		\draw [style=Arc] (17) to (14);
		\draw [style=Arc] (15) to (7);
		\draw [style=Arc] (7) to (18);
		\draw [style=Arc] (7) to (30.center);
		\draw [style=Arc] (29.center) to (7);
		\draw [style=Arc] (18) to (28.center);
		\draw [style=Arc] (27.center) to (18);
		\draw [style=Arc] (17) to (24.center);
		\draw [style=Arc] (23.center) to (17);
		\draw [style=Arc] (14) to (22.center);
		\draw [style=Arc] (21.center) to (14);
		\draw [style=Arc] (63) to (62);
		\draw [style=Arc] (63) to (64);
		\draw [style=Arc] (64) to (17);
		\draw [style=Arc] (17) to (63);
		\draw [style=Arc] (63) to (15);
		\draw [style=Arc] (7) to (63);
		\draw [style=Arc] (62) to (7);
		\draw [style=Arc] (18) to (62);
		\draw [style=Arc] (62) to (57.center);
		\draw [style=Arc] (74.center) to (7);
		\draw [style=Arc] (8) to (73.center);
		\draw [style=Arc] (72.center) to (16);
		\draw [style=Arc] (16) to (71.center);
		\draw [style=Arc] (70.center) to (15);
		\draw [style=Arc] (12) to (69.center);
		\draw [style=Arc] (68.center) to (12);
		\draw [style=Arc] (14) to (67.center);
		\draw [style=Arc] (66.center) to (20.center);
		\draw [style=Arc] (85.center) to (86.center);
		\draw [style=Arc] (95) to (96);
		\draw [style=Arc] (97) to (98);
		\draw [style=Arc] (98) to (99);
		\draw [style=Arc] (99) to (97);
		\draw [style=Arc] (99) to (100);
		\draw [style=Arc] (100) to (96);
		\draw [style=Arc] (96) to (99);
		\draw [style=Arc] (99) to (101);
		\draw [style=Arc] (101) to (98);
		\draw [style=Arc] (99) to (95);
		\draw [style=Arc] (95) to (102);
		\draw [style=Arc] (95) to (111.center);
		\draw [style=Arc] (110.center) to (95);
		\draw [style=Arc] (102) to (109.center);
		\draw [style=Arc] (108.center) to (102);
		\draw [style=Arc] (101) to (107.center);
		\draw [style=Arc] (106.center) to (101);
		\draw [style=Arc] (98) to (105.center);
		\draw [style=Arc] (104.center) to (98);
		\draw [style=Arc] (114) to (113);
		\draw [style=Arc] (114) to (115);
		\draw [style=Arc] (115) to (101);
		\draw [style=Arc] (101) to (114);
		\draw [style=Arc] (114) to (99);
		\draw [style=Arc] (95) to (114);
		\draw [style=Arc] (113) to (95);
		\draw [style=Arc] (102) to (113);
		\draw [style=Arc] (115) to (125.center);
		\draw [style=Arc] (126.center) to (115);
		\draw [style=Arc] (115) to (127.center);
		\draw [style=Arc] (128.center) to (114);
		\draw [style=Arc] (114) to (129.center);
		\draw [style=Arc] (130.center) to (114);
		\draw [style=Arc] (113) to (131.center);
		\draw [style=Arc] (132.center) to (113);
		\draw [style=Arc] (113) to (133.center);
		\draw [style=Arc] (136.center) to (113);
		\draw [style=Arc] (137.center) to (64);
		\draw [style=Arc] (64) to (103.center);
		\draw [style=Arc] (98) to (64);
		\draw [style=Arc] (64) to (97);
		\draw [style=Arc] (97) to (63);
		\draw [style=Arc] (63) to (99);
		\draw [style=Arc] (100) to (63);
		\draw [style=Arc] (62) to (100);
		\draw [style=Arc] (96) to (62);
		\draw [style=Arc] (62) to (95);
		\draw [style=Arc] (134.center) to (62);
		\draw [-{angle 60},dashed] (140.center) to (92.center);
		\draw [-{angle 60},dashed] (93.center) to (143.center);
		\draw [style=Arc] (147.center) to (115);
		\draw [style=Arc] (145.center) to (146.center);
		\draw [style=Arc] (113) to (144.center);
		\draw [-{angle 60},dotted, line width=0.8pt] (93.center) to (92.center);
		\draw [-{angle 60},dotted, line width=0.8pt] (143.center) to (140.center);
	\end{pgfonlayer}
\end{tikzpicture}}        
	\caption{ A 4-chromatic triangulation of the Klein bottle together with an orientation of the faces, whose dual is balanced.
	}
	\label{fig:KleinBottleCounterExample}
\end{figure}

\section{Winding number and circular coloring of signed graphs}\label{sec:Widing-Circular-Signed}

For a real number $r$, a circular $r$-coloring of a signed graph $(G,\sigma)$, if it exists, is a mapping $\phi$ of its vertices to the circle $\cal O_r$ in such a way that if the vertices $x$ and $y$ are adjacent with a negative edge, then $\phi(x)$ and $\phi(y)$ are at distance at least $1$ and if they are adjacent with a positive edge, then the distance of $\phi(x)$ from $-\phi(y)$ (the antipodal of $\phi(y)$) is at least $1$. The smallest value of $r$ for which $(G, \sigma)$ admits a circular $r$-coloring is the \emph{circular chromatic number} of $(G,\sigma)$, denoted $\chi_c(G,\sigma)$. The notion was first defined in \cite{NWZ21} but we have reversed the conditions on the positive and negative edges for better suitability with the theory of balanced coloring. Indeed, the circular chromatic number is a refinement of the balanced chromatic number: $$\chi_b(G,\sigma)=\left\lceil \frac{\chi_c(G,\sigma)}{2} \right\rceil.$$  We note that the circular chromatic number of a signed graph where all edges are negative is the same as the circular chromatic number of its underlying graph. But its balanced chromatic number is $\left \lceil \frac{\chi(G)}{2} \right \rceil$.  We refer to \cite{JMNNQ26} for more on the subject.

Our goal here is to present similar results for signed graphs. Most interesting parameters in the study of signed graphs are those invariant under switching operation.  For a mapping $\phi$ of the vertices of $(G,\sigma)$ to $\cal O$ this is attained by imposing the condition $\phi(-x) =-\phi(x)$. Recall that in graphs (with no signature), a mapping of the vertices of a $k$-cycle $C$ to $\cal O$ leads to $2^k$ possible linear extensions. For a signed cycle $(C, \sigma)$ the possibilities increase because for a linear extension of an edge $xy$ we allow $-\phi(x)$, $-\phi(y)$ to be considered as endpoints as well as $\phi(x)$ and $\phi(x)$. Once again we will consider two special linear extensions of a mapping $\phi$ (illustrated in \Cref{fig:signedClockShortExtension}). In order for them to be well defined we assume the following conditions on $\phi$: the two ends of a positive edge are not mapped to a pair of antipodals points and the two ends of a negative edge are not mapped to the same point. 

\begin{itemize}
	\item The clockwise extension of $\phi$. We note again that when traveling along an edge $x_{i-1}x_{i}$, we may end up at $-\phi(x_{i})$ rather than $\phi(x_{i})$. Assuming we have ended up at $\psi(x_i)\in \{\phi(x_i), -\phi(x_i)\}$, to extend $x_{i}x_{i+1}$, starting from $\psi(x_{i})$ we travel on the clockwise direction of $\cal O$ to $\psi(x_{i+1})$ where  $\psi(x_{i+1})\in \{\phi(x_{i+1}), -\phi(x_{i+1})\}$ is chosen in such a way that the product of the signs of $\psi(x_i)$ and $\psi(x_{i+1})$ is opposite of the sign of the edge $x_{i}x_{i+1}$.

	\item The shortest extension of $\phi$. To extend $x_ix_{i+1}$, starting from $\psi(x_{i})$, we travel on the shorter side of $\cal O$ to $\psi(x_{i+1})$ where $\psi(x_{i+1}) \in \{\phi(x_{i+1}), -\phi(x_{i+1})\}$ is chosen in such a way that the product of the signs of $\psi(x_{i})$, $\psi(x_{i+1})$, and $\sigma(x_ix_{i+1})$ is positive. 
\end{itemize}

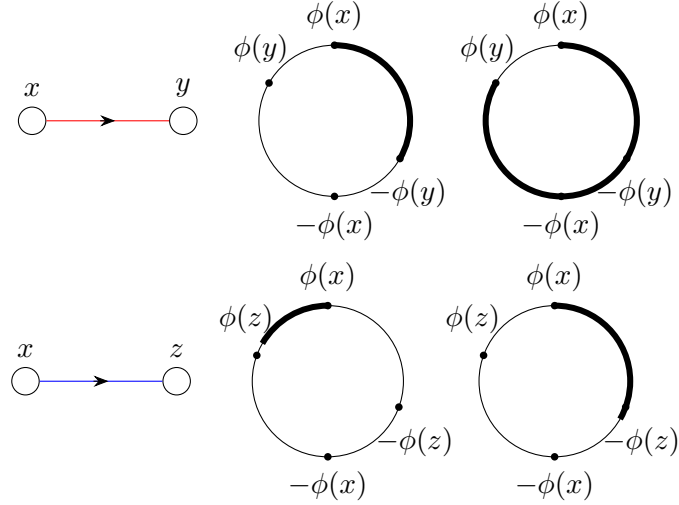
\begin{figure}[h]
\centering
\begin{tikzpicture}[scale=1.]
\def\r{1}

\node[fill=white, draw=black, shape=circle,label=above:$x$]   (xnode)  at (0,0) {};
\node[fill=white, draw=black, shape=circle,label=above:$y$]   (ynode)  at (2,0) {};
\draw[red] (xnode) -- (ynode);

\draw[-{Stealth[length=2mm]}] (0.9,0) -- ++ (0.2,0);   

\coordinate (C) at (4,0);
\draw (C) circle (\r);
\node[circle,fill,inner sep=1pt,label=above:$\phi(x)$]   (x)  at ($(C)+(90:\r)$) {};
\node[circle,fill,inner sep=1pt,label=below:$-\phi(x)$] (-x) at ($(C)+(180+90:\r)$) {};
\node[circle,fill,inner sep=1pt,label={[xshift=-3pt,yshift=2pt]above:$\phi(y)$}]  (y)  at ($(C)+(150:\r)$) {};
\node[circle,fill,inner sep=1pt,label={[xshift=3pt,yshift=-2pt]below:$-\phi(y)$}]  (-y) at ($(C)+(180+150:\r)$) {};

\draw[line width=.8mm] (x) arc[start angle=90,end angle=150-180,radius=1];

\coordinate (D) at (7,0);
\draw (D) circle (\r);
\node[circle,fill,inner sep=1pt,label=above:$\phi(x)$]   (x)  at ($(D)+(90:\r)$) {};
\node[circle,fill,inner sep=1pt,label=below:$-\phi(x)$] (-x) at ($(D)+(180+90:\r)$) {};
\node[circle,fill,inner sep=1pt,label={[xshift=-3pt,yshift=2pt]above:$\phi(y)$}]  (y)  at ($(D)+(150:\r)$) {};
\node[circle,fill,inner sep=1pt,label={[xshift=3pt,yshift=-2pt]below:$-\phi(y)$}]  (-y) at ($(D)+(180+150:\r)$) {};

\draw[line width=.8mm] (x) arc[start angle=90,end angle=150-360,radius=1];
\end{tikzpicture}

\begin{tikzpicture}[scale=1.]
\def\r{1}

\node[fill=white, draw=black, shape=circle,label=above:$x$]   (xnode)  at (0,0) {};
\node[fill=white, draw=black, shape=circle,label=above:$z$]   (ynode)  at (2,0) {};
\draw[blue] (xnode) -- (ynode);

\draw[-{Stealth[length=2mm]}] (0.9,0) -- ++ (0.2,0);   

\coordinate (C) at (4,0);
\draw (C) circle (\r);
\node[circle,fill,inner sep=1pt,label=above:$\phi(x)$]   (x)  at ($(C)+(90:\r)$) {};
\node[circle,fill,inner sep=1pt,label=below:$-\phi(x)$] (-x) at ($(C)+(180+90:\r)$) {};
\node[circle,fill,inner sep=1pt,label={[xshift=-4pt,yshift=3pt]above:$\phi(z)$}]  (y)  at ($(C)+(160:\r)$) {};
\node[circle,fill,inner sep=1pt,label={[xshift=6pt,yshift=-2pt]below:$-\phi(z)$}]  (-y) at ($(C)+(180+160:\r)$) {};

\draw[line width=.8mm] (x) arc[start angle=90,end angle=150,radius=1];

\coordinate (D) at (7,0);
\draw (D) circle (\r);
\node[circle,fill,inner sep=1pt,label=above:$\phi(x)$]   (x)  at ($(D)+(90:\r)$) {};
\node[circle,fill,inner sep=1pt,label=below:$-\phi(x)$] (-x) at ($(D)+(180+90:\r)$) {};
\node[circle,fill,inner sep=1pt,label={[xshift=-4pt,yshift=3pt]above:$\phi(z)$}]  (y)  at ($(D)+(160:\r)$) {};
\node[circle,fill,inner sep=1pt,label={[xshift=6pt,yshift=-2pt]below:$-\phi(z)$}]  (-y) at ($(D)+(180+160:\r)$) {};

\draw[line width=.8mm] (x) arc[start angle=90,end angle=150-180,radius=1];

\end{tikzpicture}
\caption{Shortest (left) and clockwise (right) extension of a positive and negative edge $xy$}
\label{fig:signedClockShortExtension}
\end{figure}

One shall note that, since we are allowed to use $-\phi(x)$ along with $\phi(x)$, when traveling along edges of a cycle, starting at $\phi(x_1)$, we may end up at $-\phi(x_1)$. Thus working with winding number of signed graphs and signed cycles, we allow winding number to be half of an integer. One may overcome this by projecting the circle to one of half size, but here we rather allow half integers.

In \Cref{fig:ExtendingSignedCycles} two examples of a linear extensions are given. Each signed cycle in these examples is obtained from the other by switching at a vertex of the other.  In the figures the vertices of the cycles are mapped to points in white. The antipodals of these points, in purple, are used to the define the linear extensions of the mappings when starting from the vertex on the top. The curves in purple present these linear extensions with the associated directions. The two curves are identical, presenting the fact that the curve formed by the linear extension is independent of the switching operation as long $\phi$ is modified according to the switching. That is to say, if a switching applied at $v$, then  $\phi$ is also modified to map $v$ to $-\phi(v)$.
\begin{figure}

\begin{minipage}[h]{.5 \textwidth}
	\hspace{0.5cm}
	\centering
	\scalebox{.5}{
		\begin{tikzpicture}

			\draw  [ black, line width=1mm, opacity=1.92] (0,0) circle (5cm);
			
			\draw node[line width=0.5mm, circle, fill=white, minimum width=1.5mm, inner sep=1mm, draw=black!80, minimum size=2mm]  (1) at ({72}:5){};
			
			\draw node[line width=0.5mm, circle, fill=white, minimum width=1.5mm, inner sep=1mm, draw=black!80, minimum size=2mm]  (2) at ({144}:5){};
			
			\draw node[line width=0.5mm, circle, fill=white, minimum width=1.5mm, inner sep=1mm, draw=black!80, minimum size=2mm]  (3) at ({216}:5){};

			\draw node[line width=0.5mm, circle, fill=white, minimum width=1.5mm, inner sep=1mm, draw=black!80, minimum size=2mm]  (4) at ({288}:5){};
			
			\draw node[line width=0.5mm, circle, fill=white, minimum width=1.5mm, inner sep=1mm, draw=black!80, minimum size=2mm]  (5) at ({0}:5){};

			\draw node[line width=0.5mm, circle, fill=purple, minimum width=1.5mm, inner sep=1mm, draw=black!80, minimum size=2mm]  (6) at ({108}:5){};
			
			\draw node[line width=0.5mm, circle, fill=purple, minimum width=1.5mm, inner sep=1mm, draw=black!80, minimum size=2mm]  (7) at ({180}:5){};
			
			\draw node[line width=0.5mm, circle, fill=purple, minimum width=1.5mm, inner sep=1mm, draw=black!80, minimum size=2mm]  (8) at ({252}:5){};

			\draw node[line width=0.5mm, circle, fill=purple, minimum width=1.5mm, inner sep=1mm, draw=black!80, minimum size=2mm]  (9) at ({324}:5){};
			
			\draw node[line width=0.5mm, circle, fill=purple, minimum width=1.5mm, inner sep=1mm, draw=black!80, minimum size=2mm]  (10) at ({36}:5){};

			\foreach \i/\j in {1/3, 5/2}{ 
				
				\draw[-{Latex[width=0mm 5,  length=4mm]} , blue, opacity=0.6, line width=0.4mm] (\i) to (\j);}
			
			\foreach \i/\j in {3/5, 2/4, 4/1}{ 
				
				\draw[-{Latex[width=0mm 5,  length=4mm]} , red, opacity=0.6, line width=0.4mm] (\i) to (\j);}

			\draw[purple,  line width=0.4mm, -{Latex[width=0mm 5,  length=4mm]}]  (1) to[out=-36, in=-36, distance=-6cm] (3);
			
			\draw[purple,  line width=0.4mm, -{Latex[width=0mm 5,  length=4mm]}]  (3) to[out=18, in=18, distance=-2cm] (7);	
			
			\draw[purple,  line width=0.4mm, -{Latex[width=0mm 5,  length=4mm]}]  (7) to[out=72, in=72, distance=-6cm] (9);	
			
			\draw[purple,  line width=0.4mm, -{Latex[width=0mm 5,  length=4mm]}]  (9) to[out=96, in=144, distance=-2cm] (4);	
			
			\draw[purple,  line width=0.4mm, -{Latex[width=0mm 5,  length=4mm]}]  (4) to[out=78, in=108, distance=-2cm] (8);	
			
			\draw node[black, line width=1mm, opacity=1.92]  () at (0,-9) {\Huge $w\curve{{\color{red}\vec C},\phi,sh}=\frac{1}{2}$};
			
		\end{tikzpicture}
	}
\end{minipage}	
\begin{minipage}[h]{.4 \textwidth}
	\hspace{1.25cm}
	\centering
	\scalebox{.5}{
		\begin{tikzpicture}

			\draw  [ black, line width=1mm, opacity=1.92] (0,0) circle (5cm);
			
			\draw node[line width=0.5mm, circle, fill=white, minimum width=1.5mm, inner sep=1mm, draw=black!80, minimum size=2mm]  (1) at ({72}:5){};
			
			\draw node[line width=0.5mm, circle, fill=white, minimum width=1.5mm, inner sep=1mm, draw=black!80, minimum size=2mm]  (2) at ({144}:5){};
			
			\draw node[line width=0.5mm, circle, fill=white, minimum width=1.5mm, inner sep=1mm, draw=black!80, minimum size=2mm]  (3) at ({216}:5){};

			\draw node[line width=0.5mm, circle, fill=purple, minimum width=1.5mm, inner sep=1mm, draw=black!80, minimum size=2mm]  (4) at ({288}:5){};
			
			\draw node[line width=0.5mm, circle, fill=white, minimum width=1.5mm, inner sep=1mm, draw=black!80, minimum size=2mm]  (5) at ({0}:5){};

			\draw node[line width=0.5mm, circle, fill=white, minimum width=1.5mm, inner sep=1mm, draw=black!80, minimum size=2mm]  (6) at ({108}:5){};
			
			\draw node[line width=0.5mm, circle, fill=purple, minimum width=1.5mm, inner sep=1mm, draw=black!80, minimum size=2mm]  (7) at ({180}:5){};
			
			\draw node[line width=0.5mm, circle, fill=purple, minimum width=1.5mm, inner sep=1mm, draw=black!80, minimum size=2mm]  (8) at ({252}:5){};

			\draw node[line width=0.5mm, circle, fill=purple, minimum width=1.5mm, inner sep=1mm, draw=black!80, minimum size=2mm]  (9) at ({324}:5){};
			
			\draw node[line width=0.5mm, circle, fill=purple, minimum width=1.5mm, inner sep=1mm, draw=black!80, minimum size=2mm]  (10) at ({36}:5){};

			\foreach \i/\j in {1/3, 2/6, 6/1, 5/2}{ 
				
				\draw[-{Latex[width=0mm 5,  length=4mm]} , blue, opacity=0.6, line width=0.4mm] (\i) to (\j);}
			
			\foreach \i/\j in {3/5}{ 
				
				\draw[-{Latex[width=0mm 5,  length=4mm]} , red, opacity=0.6, line width=0.4mm] (\i) to (\j);}

			\draw[purple,  line width=0.4mm, -{Latex[width=0mm 5,  length=4mm]}]  (1) to[out=-36, in=-36, distance=-6cm] (3);
			
			\draw[purple,  line width=0.4mm, -{Latex[width=0mm 5,  length=4mm]}]  (3) to[out=18, in=18, distance=-2cm] (7);	
			
			\draw[purple,  line width=0.4mm, -{Latex[width=0mm 5,  length=4mm]}]  (7) to[out=72, in=72, distance=-6cm] (9);	
			
			\draw[purple,  line width=0.4mm, -{Latex[width=0mm 5,  length=4mm]}]  (9) to[out=96, in=144, distance=-2cm] (4);	
			
			\draw[purple,  line width=0.4mm, -{Latex[width=0mm 5,  length=4mm]}]  (4) to[out=78, in=108, distance=-2cm] (8);	
			
			\draw node[black, line width=1mm, opacity=1.92]  () at (0,-9) {\Huge $w\curve{{\color{red}\vec C'},\phi',sh}=\frac{1}{2}$};

		\end{tikzpicture}
	}
\end{minipage}	
\caption{Continuous extension of signed cycles with shortest conventions and impact of switching}
\label{fig:ExtendingSignedCycles}
\end{figure}

We may then observe that: in the convention of shortest extension the winding number of signed cycle is an integer if and only it is positive. In the convention of clockwise extension, the winding number of a signed cycle is an integer if and only if the number of positive edges is odd.
A key property for either of the two conventions is that they are invariant under switching.

\begin{observation} 
Given a signed cycle and mapping $\phi$ of its vertices to $\mathcal{O}$, in either of the clockwise  or shortest extensions, switching at a vertex $x_i$ does not impact the choice for $\psi(x_i)$. And therefore, does not impact the winding number.
\end{observation}
\begin{proof}
	Consider an edge $x_{i-1}x_i$. Let $-x_i$ be the vertex obtained from switching at $x_i$. Recall that for a circular coloring $\phi$, we impose $\phi(-x_i) = -\phi(x_i)$. Also, the sign of the edge $x_{i-1}x_i$ is switched. Those two sign changes cancel out the impact in the choice of $\psi(x_i)$. See \Cref{fig:ExtendingSignedCycles}.
\end{proof}

\begin{observation}\label{obs:Vec-Cev}
	Given a (signed) cycle we have $w\curve{\vec C,\phi,sh}=-w\curve{\cev C,\phi,sh}$.
\end{observation}

If $\phi$ is a circular coloring of the cycle $C$, then there is a strong connection between the two conventions of clockwise extension and shortest extension. Assume starting at the point $\phi(x_i)$ in the clockwise extension, the edge $x_ix_{i+1}$ maps to an arc of length $\ell_c$. And in the shortest convention it maps to an arc of algebraic length $\ell_s$, meaning that if the shortest arc travels counterclockwise, then $\ell_s$ is a negative number. Note that the value of $\psi(x_{i+1})$ will be different between the two extensions. Then it can be readily verified that for any arc: $\ell_c=\ell_s + \frac r2 $ (see \Cref{fig:signedClockShortExtension}). Summing this formula over all edges of $C$, we have the following correlation between the two extensions.     

\begin{proposition}
	For a cycle $C$ and a circular coloring $\phi$, $w\curve{C,\phi,\clock} = w\curve{C,\phi,sh} + \frac{|C|}{2}$.
\end{proposition}

In particular if $C$ is a positive directed $2$-cycle, then $w\curve{C,\phi,sh}=0$. This is a key fact when applying winding number to circular coloring. The other key facts that help to prove lower bound on the circular chromatic number of graphs using the winding numbers of a set its cycles are the following. 

\begin{proposition}\label{prop:forced-winding}
	Let $C$ be a signed $k$-cycle and let $\phi$ be a circular $r$-coloring of $C$. 
	\begin{itemize}
	\item If $C$ is positive and $r<\frac{2k}{k-2}$, then $w\curve{\vec C,\phi,sh} = 0$.
	\item If $C$ is negative, then $r\geq \frac{2k}{k-1}$. Moreover, if $r < \frac{2k}{k-3}$, then $|w\curve{\vec C,\phi,sh}| = \frac12$.
	\item If $C$ is negative and $r<\frac{2k-2}{k-2}$, then in the shortest extension, all edges travel in the same direction a nonzero length.  
	\end{itemize}
\end{proposition}

\begin{proof}
	
	For the first item, recall that in the extension by the shortest convention, a positive edge $x_ix_{i+1}$ is extended to the shorter side of $\mathcal{O}$ partitioned by either $\phi(x_i),\phi(x_{i+1})$ or $-\phi(x_i),-\phi(x_{i+1})$. If $x_ix_{i+1}$ is negative, then it is mapped to the shorter side of $\mathcal{O}$ partitioned by either $\phi(x_i),-\phi(x_{i+1})$ or $-\phi(x_i),\phi(x_{i+1})$. The latter is the same as mapping the positive edge $x_ix_{i+1}$ after a switching at $x_{i+1}$. By the condition of circular $r$-coloring on the positive edges, each of these edges is traveled a length of at most $\frac{r}{2} -1$. But if $|w\curve{\vec C,\phi,sh}| \geq 1$, then at least one of the edges has traveled a length of at least $\frac{r}{k}$. Thus under this condition we have $\frac{r}{2}-1 \leq \frac{r}{k}$, or equivalently $r\geq \frac{2k}{k-2}$.
	
	For the second item, we apply the same conditions on the edges, but we note that the winding number of a negative cycle under the shortest convention is an integer plus a half, thus its absolute value is at least $\frac{1}{2}$. Therefore we have $\frac{r}{2}-1 \geq \frac{r}{2k}$, or equivalently $r\leq \frac{2k}{k-1}$. For the moreover part, if $|w\curve{\vec C,\phi,sh}| \neq \frac12$, then $|w\curve{\vec C,\phi,sh}| \geq \frac32$ and by the same argument we have $\frac{r}{2}-1 \geq \frac{3r}{2k}$ which means $r \geq \frac{2k}{k-3}$.

	For the third item,	the case of $k=2$, that is when $C$ is a negative $2$-cycle, the statement can be readily verified. For $k\geq 3$ we have $\frac{2k-2}{k-2} \leq \frac{2k}{k-1}$, hence by the second item, we have $|w\curve{\vec C,\phi,sh}| = \frac12$. By symmetry we assume  $w\curve{\vec C,\phi,sh} = +\frac12$. Thus, the claim to be proven is that, in the shortest extension, each edge travels a nonzero length in the clockwise direction. Assume to the contrary, that an edge is either mapped to a single point or traverses a counterclockwise direction. But then even if each of the other edges traverses the maximum allowed clockwise distance of $\frac{r}{2}-1<\frac{r}{2(k-1)}$, the total travel length will be less than $\frac{r}{2}$.
\end{proof}

We shall note that in these statements we allow  the denominator to be $0$, in which case the  corresponding conclusion holds for each value of $r$.  

\subsection{Positive cycles}

\begin{theorem}\label{thm:setOfEvenCyclesSigned}
		Let $\mathcal{C}$ be a set of positive cycles of a signed graph $\widehat G$ each of length at most $k$, where each edge is in an even number of cycles of $\mathcal{C}$. If for some assignment $D$ of directions to the cycles in $\mathcal{C}$, the sum of the  $(\mathcal{C}, D)$-degrees of the negative edges is $2 \pmod 4$, then $\chi_c(\widehat G) \geq \frac{2k}{k-2}$. 
	\end{theorem}
\begin{proof}
	We consider the oriented signed multigraph $G_1$, obtained from $\mathcal{C}$ by taking the edge disjoint union of the (directed) signed cycles, but allowing the same vertices to be identified. As a sum of (directed) cycles, observe that $G_1$ is eulerian. Let $\mathcal{C}'$ be a maximal set of edge disjoint directed $2$-cycles in $G_1$, and let $G'_1$ be the subgraph of $G_1$ obtained by removing the edges of all 2-cycles in $\mathcal{C}'$.

	Since every edge appears in an even number of cycles of $\cal C$, in $G'_1$ each arc is in a class of an even number of parallel arcs. As $G'_1$ is still Eulerian, its edges can be decomposed into a set $\mathcal{C}''$ of edge-disjoint (directed) cycles. In process of this decomposition, each time a cycle is selected, we also select a copy of it, noting that, after removing the two, what remains satisfies the conditions of being eulerian and each edge being in a parallel class of an even size.

	Let $\phi$ be a $r$-circular coloring of $\widehat G$. Towards a contradiction, suppose $r < \frac{2k}{k-2}$, by \Cref{prop:forced-winding}, for every cycle of length at most $k$, $w\curve{\vec C, \phi, sh} = 0$. Thus 
	\begin{align*}
	0 &= \sum_{\vec C\in \cal C}w \curve{\vec C, \phi, sh} \\
	&= \sum_{\vec C\in  \cal C' \cup \cal C''}w \curve{\vec C, \phi, sh}\\
	&= \sum_{\vec C\in \cup \cal C''}w \curve{\vec C, \phi, sh}  \quad \text{because } w \curve{\vec C_2, \phi, sh} = 0\\
	&= 	\sum_{\substack{ \vec C\in \cal C''\\ C \text{ positive}}} w \curve{\vec C, \phi, sh} + 
	\sum_{\substack{ \vec C\in \cal C''\\ C \text{ negative}}}w \curve{\vec C, \phi, sh} &(*)
	\end{align*}
	
	Let's $C$ be a cycle of $\cal C''$. 
	\begin{itemize}
	
	\item If $C$ is positive, then $w\curve{C,\phi,sh} = \ell$ with $\ell$ being an integer. Since $C$ is taken twice, this pair contributes $2\ell \equiv 0 \mod 2$ in the sum $(*)$. And because it has an even number of negative edges, the pair contributes $0 \mod 4$ to the sum of the $(\mathcal{C}, D)$-degrees of negative edges.

	\item If $C$ is negative, then $w\curve{C,\phi,sh} = \ell + \frac{1}{2}$ with $\ell$ being an integer. Since $C$ is taken twice, it contributes $2\ell + 1 \equiv 1 \mod 2$ in the sum $(*)$. And because it has an odd number of negative edges, the pair contributes $2 \mod 4$ to the sum of the the $(\mathcal{C}, D)$-degrees of negative edges.
	\end{itemize}
	
	From the winding number, we deduce that there must be an even number of pairs of negative cycles in $\cal C''$. However, because the sum of the $(\mathcal{C}, D)$-degrees of the negative edges being $2 \mod 4$, we need an odd number of pairs of negative cycles. Thus we reach a contradiction. 
	\end{proof}

\subsection{Negative cycles}

Given a signed graph $\widehat{G}$, a set $\cal C$ of its cycles, and directions $D$ on the cycles in $\cal C$, the dual of the system is defined in the same way as in the graph case. Thus the signs of the edges in the dual are only function of the orientations assigned by $D$.

 \begin{theorem}
 	Let $\widehat G$ be a graph and $\mathcal{C}$ be a set of negative $k$-cycles of $G$. If $\widehat S(\cal C)$ is not balanced, then $\chi_c(\widehat G)\geq \frac{2k-2}{k-2}$. 
 \end{theorem}

 \begin{proof}
Suppose that $\chi_c(\widehat G) < \frac{2k-2}{k-2}$. By the third item of \Cref{prop:forced-winding} we have $w\curve{\vec C,\phi,sh} \in \{-\frac{1}{2}, +\frac{1}{2}\}$.
By \Cref{obs:Vec-Cev} we may choose a direction for each cycle such that the corresponding winding number is $+\frac{1}{2}$. Again by the third item of \Cref{prop:forced-winding} each edge travels a non-zero length in the clockwise direction. In particular, the orientation is forced, so all cycles must agree on the orientation of their common edges. This is translated in $\widehat S(\cal C, D)$ to all edges being positive. So $\widehat S(\cal C)$ is balanced.
 \end{proof}

\subsection{Signed graphs on surfaces}

In applying the result to a graph embedded on a surface we normally take a set $\mathcal{C}$ of cycles which consists of all facial cycles plus a few noncontractible cycles. The directions associated to these cycles will be the directions through which noncontractible cycles are cut to form a planar embedding after which each face will be assigned a clockwise direction. We omit details of the proofs when they are straightforward.

\begin{theorem}\label{thm:SignProjectivePlanar-positive}
	Let $\widehat{G}$ be a signed graph embedded on the projective plane such that each face is a positive cycle of length $k$ ($k\geq 3$). Then either $\widehat{G}$ is balanced, or $\chi_c(G) \geq \frac{2k}{k-2}$.   
\end{theorem}

\begin{corollary}\label{cor:triangulation}
	If $\widehat{G}$ is a signed triangulation of the projective plane where each facial triangle is positive, then either $\chi_c(\widehat{G})=2$ (when $\widehat{G}$ is balanced) or $\chi_c(\widehat{G})=6$. 
\end{corollary}

\begin{proof}
	If $\widehat{G}$ is not balanced, then, by \Cref{thm:SignProjectivePlanar-positive}, we have $\chi_c(\widehat{G})\geq 6$. But on the other hand, by the Euler formula, we know $G$, the underlying graph of $\widehat{G}$, is 5-degenerate, and hence admits a circular 6-coloring (see \cite{NWZ21} for more details).
\end{proof}

Similarly we have the following.

 \begin{corollary}\label{cor:SignedQuad}
 	If $\widehat{G}$ is a signed quadrangulation of the projective plane where each face is positive, then either $\chi_c(\widehat{G})=2$ (when $\widehat{G}$ is balanced) or $\chi_c(\widehat{G})=4$. 
 \end{corollary}
 
 \begin{proof}
 	If $\widehat{G}$ is not balanced, then, by \Cref{thm:SignProjectivePlanar-positive}, we have $\chi_c(\widehat{G})\geq 4$. To complete the proof, it remains to show that any such signed graph admits a circular $4$-coloring. To that end it would be enough to show that they are 3-degenerate. It follows from the Euler formula ($n-e+f=1$) that if each face of $\widehat{G}$ is of length at least 4, then it has a vertex of degree at most 3. As $\widehat{G}$ is a quadrangulation, it has no contractible triangle, because any such a triangle should contain a face of odd length in its inner planar part. Thus for any subgraph $G'$ of $\widehat{G}$, in the embedding induced by the embedding of $G$ of $G$ in the projective plane all  faces are of length at least 4. Therefore, by the Euler formula, $G'$ has a vertex of degree at most 3.  
 \end{proof}
 
By taking graphs as signed graphs where all edges are negative, one observes that this extends the results on graphs mentioned in the introduction. The parallel case of signed qudrangulations of the projective plane whose underlying graphs are bipartite was introduced in \cite{N22} and is shown to imply the case of graphs.

We shall note that \Cref{cor:triangulation} which is essentially the first case of \Cref{thm:SignProjectivePlanar-positive} ($k=3$), is the most crucial case. We show for example how the case $k=4$ would be implied by the case $k=3$. 

If $\phi$ is a circular $(4-\epsilon)$-coloring of a $4$-cycles $xyzt$ where all edges are positive, then one of the two diagonals, $xz$ or $yt$, say $xz$, is bounded inside one of the positive edges. That means, if we add $xz$ as a positive edge, then $\phi$ remains a valid circular $(4-\epsilon)$-coloring. If start with a positive $4$-cycle with some negative edges, then we $xz$ with a signature such that if we switch to make the $4$-cycle all positive, then $xz$ is also positive. Now, given a signed graph $\widehat{G}$ which is not balanced and embedded on the projective plane where all faces are positive 4-cycles, if there is a circular $(4-\epsilon)$-coloring $\phi$ of $\widehat{G}$, then we may add an edge inside each face to obtain a suppergraph also embedded on the projective plane where each face is a positive triangle. The selection of these edge guarantees that $\phi$ is a circular $(4-\epsilon)$-coloring of the larger graph, but the circular chromatic number of this graph is 6.

The case where all faces are negative cycles of a given length extends the case of graphs where all faces are odd cycles of a given length.

\begin{theorem}\label{thm:SignProjectivePlanar-negative}
	Let $\widehat{G}$ be a signed graph embedded on the projective plane such that each face is a negative cycle of length $k$. Then either $\chi_c(\widehat{G})=\frac{2k}{k-1}$, or $\chi_c(G) \geq \frac{2k-2}{k-2}$.   
\end{theorem}

In other words, a signed projective planar graph where all faces are negative cycles of length $k$ either maps to $C_{-k}$ or its circular chromatic number is at least as that of $C_{-(k-1)}$.

We should remark that to find a best possible upper bound for the circular chromatic number of signed planar graphs of negative girth at least $k$ is difficult question which capture in special case the Jaeger-Zhang conjecture and relates to study of flows, we refer to \cite{LNWZ24} for more on this subject.

\section{Examples}\label{sec:Examples}

Here we first present essential families of signed graphs embedded on the projective plane where every face is a positive $4$-cycle, they are not balanced, and for each integer $k$ there is a signed graphs in the family whose shortest negative cycle is of length at least $k$. It is expected that signed graphs of negative girth at least $k$ that do not admit circular $3$-coloring and have smallest number of vertices occur among these examples.

Next, by adding edges to the quadrangulations, we build signed graphs of arbitrarily high negative girth with embedding on the projective plane such that each face is a positive triangle. Their circular chromatic number then would be $6$. 

\subsection{Quadrangulations of the projective plane}

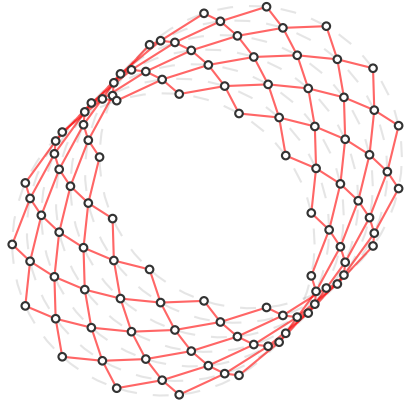
\begin{figure}[!htb]
	\centering
	\scalebox{2}{		\tikzset{math3d/.style= {x={(2cm,0cm)}, y={(0cm,2cm) }}} 
			
				\begin{tikzpicture}[math3d]
					\newcommand{\n}{15}
					\newcommand{\h}{7}
					\newcommand{\ch}{12}
					\newcommand{\rl}{1}
					\newcommand{\rh}{1}
					
					\foreach \i in {1,...,\h}{
						\path[draw, dashed, gray, opacity=.2] plot[domain=0:2*pi,samples=4*\n] ({\rh*cos(\x r)}, {\rh*sin(\x r)}, .5*\i);
						
					}

					\foreach \i in {1,3,5}{
						\foreach \t in {1,...,\n} {
							
							\draw [ thin, red, opacity=.6] ( ({(\rl)*cos((2*\t+1)*pi/\n r)},{\rl*sin((2*\t+1)*pi/\n r)},.5*\i) -- ({\rh*cos((2*\t+1)*pi/\n r-\ch)},{\rh*sin((2*\t+1)*pi/\n r-\ch)},.5*\i+.5)  -- cycle;
							
							\draw [ thin, red, opacity=.6] ( ({(\rl)*cos((2*\t+1)*pi/\n r)},{\rl*sin((2*\t+1)*pi/\n r)},.5*\i) -- ({\rh*cos((2*\t+1)*pi/\n r+\ch)},{\rh*sin((2*\t+1)*pi/\n r+\ch)},.5*\i+.5)  -- cycle;
						}

					}
					
					\foreach \i in {2,4,6}{
						\foreach \t in {1,...,\n} {
							
							\draw [ thin, red, opacity=.6] ( ({(\rl)*cos((2*\t+1)*pi/\n r+\ch)},{\rl*sin((2*\t+1)*pi/\n r+\ch)},.5*\i) -- ({\rh*cos((2*\t+1)*pi/\n r)},{\rh*sin((2*\t+1)*pi/\n r)},.5*\i+.5)  -- cycle;
							
							\draw [ thin, red, opacity=.6] ( ({(\rl)*cos((2*\t+1)*pi/\n r-\ch)},{\rl*sin((2*\t+1)*pi/\n r-\ch)},.5*\i) -- ({\rh*cos((2*\t+1)*pi/\n r)},{\rh*sin((2*\t+1)*pi/\n r)},.5*\i+.5)  -- cycle;
						}
					} 
					
					\foreach \i in {1,3,5,7}{

						\foreach \t in {1,...,\n} {
							\node[circle, black, inner sep=0mm, draw=black!80, fill=white, minimum size=.5mm] at  ({(\rl)*cos((2*\t+1)*pi/\n r)},{\rl*sin((2*\t+1)*pi/\n r)},.5*\i){};
						}
					}
					
					\foreach \i in {2,4,6}{

						\foreach \t in {1,...,\n} {
							\node[circle, red, inner sep=0mm, draw=black!80, fill=white, minimum size=.5mm] at  ({(\rl)*cos((2*\t)*pi/\n r)},{\rl*sin((2*\t)*pi/\n r)},.5*\i){};
						}
					}

				\end{tikzpicture}}        
	\caption{$P_{\ell}\times(C_{2k+1})$}
	\label{fig:PlTimesC2k+1}
\end{figure}

The graphs $M_k(C_{2k+1})$ are most prominent examples of 4-critical graphs studied by various authors. The graph  $M_{\ell}(C_{2k+1})$, for any pair of positive integers $l$ and $k$, is built from the categorical product $P_{\ell} \times C_{2k+1}$ viewed as presented in \Cref{fig:PlTimesC2k+1} by adding a universal vertex to one end and near antipodal edges to the other end as depicted in \Cref{fig:M_l(C_{2k+1})}.

\begin{figure}[!htb]
	\centering
	\scalebox{1.5}{	
	\begin{minipage}{.4\textwidth}

		\scalebox{1.3}{
			\tikzset{math3d/.style= {x={(2cm,0cm)}, y={(0cm,2cm)}}}
			\begin{tikzpicture}[math3d]
				\newcommand{\n}{15}
				\newcommand{\h}{7}
				\newcommand{\ch}{12}
				\newcommand{\rl}{1}
				\newcommand{\rh}{1}
				
				
				\path[draw, gray] plot[domain=0:2*pi,samples=4*\n] ({\rh*cos(\x r)}, {\rh*sin(\x r)}, 3.5);
				
				\foreach \i in {1,...,6}{
					\path[draw, dashed, gray] plot[domain=0:2*pi,samples=4*\n] ({\rh*cos(\x r)}, {\rh*sin(\x r)}, .5*\i);
				}
				
				%
				
				\foreach \i in {1,3,5}{
					\foreach \t in {1,...,\n} {
						
						\draw [thick, red, opacity=0.2] ( ({(\rl)*cos((2*\t+1)*pi/\n r)},{\rl*sin((2*\t+1)*pi/\n r)},.5*\i) -- ({\rh*cos((2*\t+1)*pi/\n r-\ch)},{\rh*sin((2*\t+1)*pi/\n r-\ch)},.5*\i+.5)  -- cycle;
						
						\draw [ thick, red, opacity=0.2] ( ({(\rl)*cos((2*\t+1)*pi/\n r)},{\rl*sin((2*\t+1)*pi/\n r)},.5*\i) -- ({\rh*cos((2*\t+1)*pi/\n r+\ch)},{\rh*sin((2*\t+1)*pi/\n r+\ch)},.5*\i+.5)  -- cycle;
					}

				}
				
				\foreach \i in {2,4,6}{
					\foreach \t in {1,...,\n} {
						
						\draw [ thick, red, opacity=0.2] ( ({(\rl)*cos((2*\t+1)*pi/\n r+\ch)},{\rl*sin((2*\t+1)*pi/\n r+\ch)},.5*\i) -- ({\rh*cos((2*\t+1)*pi/\n r)},{\rh*sin((2*\t+1)*pi/\n r)},.5*\i+.5)  -- cycle;
						
						\draw [ thick, red, opacity=0.2] ( ({(\rl)*cos((2*\t+1)*pi/\n r-\ch)},{\rl*sin((2*\t+1)*pi/\n r-\ch)},.5*\i) -- ({\rh*cos((2*\t+1)*pi/\n r)},{\rh*sin((2*\t+1)*pi/\n r)},.5*\i+.5)  -- cycle;
					}
				}

				\foreach \t in {1,...,\n} {
					
					\draw [  red, opacity=1] ( ({(cos((2*\t+14)*pi/\n r+\ch)},{sin((2*\t+14)*pi/\n r+\ch)},3.5) -- (-.5,-.5)  -- cycle;
					
				}
				\draw node[circle, red, inner sep=0mm, draw=black!80, fill=white, minimum size=1mm] at (-.5,-.5) {};
			\end{tikzpicture}
		}
	\end{minipage}
	\begin{minipage}{.4\textwidth}
		\scalebox{1.3}{
			\tikzset{math3d/.style= {x={(2cm,0cm)}, y={(0cm,2cm)}}}
			\begin{tikzpicture}[math3d]
				\newcommand{\n}{15}
				\newcommand{\h}{7}
				\newcommand{\ch}{12}
				\newcommand{\rl}{1}
				\newcommand{\rh}{1}
				
				\path[draw, gray] plot[domain=0:2*pi,samples=4*\n, opacity=.5] ({\rh*cos(\x r)}, {\rh*sin(\x r)}, .5);
				
				\foreach \i in {1,...,\h}{
					\path[draw, dashed, gray] plot[domain=0:2*pi,samples=4*\n] ({\rh*cos(\x r)}, {\rh*sin(\x r)}, .5*\i);
					
					%
					
				}
				\foreach \i in {1,3,5}{
					\foreach \t in {1,...,\n} {
						
						\draw [thick, red, opacity=0.2] ( ({(\rl)*cos((2*\t+1)*pi/\n r)},{\rl*sin((2*\t+1)*pi/\n r)},.5*\i) -- ({\rh*cos((2*\t+1)*pi/\n r-\ch)},{\rh*sin((2*\t+1)*pi/\n r-\ch)},.5*\i+.5)  -- cycle;
						
						\draw [ thick, red, opacity=0.2] ( ({(\rl)*cos((2*\t+1)*pi/\n r)},{\rl*sin((2*\t+1)*pi/\n r)},.5*\i) -- ({\rh*cos((2*\t+1)*pi/\n r+\ch)},{\rh*sin((2*\t+1)*pi/\n r+\ch)},.5*\i+.5)  -- cycle;
					}

				}
				
				\foreach \i in {2,4,6}{
					\foreach \t in {1,...,\n} {
						
						\draw [ thick, red, opacity=0.2] ( ({(\rl)*cos((2*\t+1)*pi/\n r+\ch)},{\rl*sin((2*\t+1)*pi/\n r+\ch)},.5*\i) -- ({\rh*cos((2*\t+1)*pi/\n r)},{\rh*sin((2*\t+1)*pi/\n r)},.5*\i+.5)  -- cycle;
						
						\draw [ thick, red, opacity=0.2] ( ({(\rl)*cos((2*\t+1)*pi/\n r-\ch)},{\rl*sin((2*\t+1)*pi/\n r-\ch)},.5*\i) -- ({\rh*cos((2*\t+1)*pi/\n r)},{\rh*sin((2*\t+1)*pi/\n r)},.5*\i+.5)  -- cycle;
					}
				}

				\foreach \t in {1,...,\n} {
					
					\draw [ red, opacity=1] ( ({(cos((2*\t+14)*pi/\n r+\ch)},{sin((2*\t+14)*pi/\n r+\ch)},.5) -- ({cos((2*\t+1)*pi/\n r)},{sin((2*\t+1)*pi/\n r)},.5)  -- cycle;
					
				}
				
			\end{tikzpicture}
		}
	\end{minipage}}        
	\caption{$M_l(C_{2k+1})$}
	\label{fig:M_l(C_{2k+1})}
\end{figure}
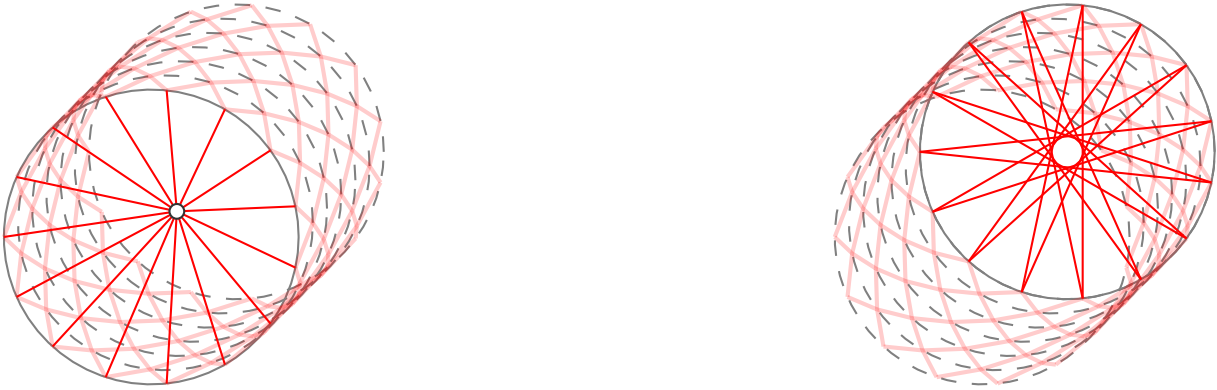

The product $P_{\ell}\times C_{2k+1}$ can be viewed as a planar graph where all faces except two, say $F_1$ and $F_{\ell}$, are 4-faces. The faces $F_1$ and $F_{\ell}$ are each a $(4k+2)$- cycle. Adding a universal vertex to the bottom layer corresponds to adding a vertex $u$ on the face $F_1$ which is joined to every second vertex of $F_1$. Thus all resulting faces are 4-cycles. On the face $F_{\ell}$, having labeled vertices in cyclic order such that vertices labeled even are further away from $u$, each vertex with an even label is connected to two vertices of even label that, on the face $F_{\ell}$, are furthest away from it. These edges can be viewed as non-crossing if a cross-cap inserted in the face after which all the resulting faces are 4-cycles. In \Cref{fig:M3-7} we have presented a projective planar embedding of $M_{3}(C_7)$.

\begin{figure}[!htb]
	\centering
	\scalebox{1}{		\centering
		\begin{tikzpicture}
				
				\foreach \i in {1,2,3,4,5,6,7}
				{ \draw node[line width=0.05mm, circle, fill=white, minimum width=.5mm, inner sep=0mm, draw=black!80, minimum size=1mm]  (\i) at ({\i*360/7}:1.2){};}
				
				\foreach \i in {8,9,10,11,12,13,14}
				{ \draw node[line width=0.05mm, circle, fill=white, minimum width=.5mm, inner sep=0mm, draw=black!80, minimum size=1mm]  (\i) at ({\i*360/7+180/7}:2){};
				}
				
				\foreach \i in {15,16,17,18,19,20,21}
				{ \draw node[line width=0.05mm, circle, fill=white, minimum width=.5mm, inner sep=0mm, draw=black!80, minimum size=1mm]  (\i) at ({\i*360/7}:2.8){};}
				
				\foreach \i in {22,23,24,25,26,27,28}
				{ \draw node[line width=0.05mm, circle, fill=white, minimum width=.5mm, inner sep=0mm, draw=black!80, minimum size=1mm]  (\i) at ({\i*360/7+180/7}:3.6){};}
				
				\foreach \i in {29,30,...,42}
				{ \draw node[line width=0.05mm, circle, fill=white, minimum width=.5mm, inner sep=0mm, draw=black!80, minimum size=0mm]  (\i) at ({\i*360/14+180/14}:4.6){};}

				\draw [ thick, red, line width=0.15mm, opacity=0.9] (1) -- (8);
				\draw [ thick, red, line width=0.15mm, opacity=0.9] (2) -- (9);
				\draw [ thick, red, line width=0.15mm, opacity=0.9] (3) -- (10);
				\draw [ thick, red, line width=0.15mm, opacity=0.9] (4) -- (11);
				\draw [ thick, red, line width=0.15mm, opacity=0.9] (5) -- (12);
				\draw [ thick, red, line width=0.15mm, opacity=0.9] (6) -- (13);
				\draw [ thick, red, line width=0.15mm, opacity=0.9] (7) -- (14);
				
				\draw [ thick, red, line width=0.15mm, opacity=0.9] (1) -- (14);
				\draw [ thick, red, line width=0.15mm, opacity=0.9] (2) -- (8);
				\draw [ thick, red, line width=0.15mm, opacity=0.9] (3) -- (9);
				\draw [ thick, red, line width=0.15mm, opacity=0.9] (4) -- (10);
				\draw [ thick, red, line width=0.15mm, opacity=0.9] (5) -- (11);
				\draw [ thick, red, line width=0.15mm, opacity=0.9] (6) -- (12);
				\draw [ thick, red, line width=0.15mm, opacity=0.9] (7) -- (13);

				\draw [ thick, red, line width=0.15mm, opacity=0.9] (8) -- (15);
				\draw [ thick, red, line width=0.15mm, opacity=0.9] (9) -- (16);
				\draw [ thick, red, line width=0.15mm, opacity=0.9] (10) -- (17);
				\draw [ thick, red, line width=0.15mm, opacity=0.9] (11) -- (18);
				\draw [ thick, red, line width=0.15mm, opacity=0.9] (12) -- (19);
				\draw [ thick, red, line width=0.15mm, opacity=0.9] (13) -- (20);
				\draw [ thick, red, line width=0.15mm, opacity=0.9] (14) -- (21);

				\draw [ thick, red, line width=0.15mm, opacity=0.9] (8) -- (16);
				\draw [ thick, red, line width=0.15mm, opacity=0.9] (9) -- (17);
				\draw [ thick, red, line width=0.15mm, opacity=0.9] (10) -- (18);
				\draw [ thick, red, line width=0.15mm, opacity=0.9] (11) -- (19);
				\draw [ thick, red, line width=0.15mm, opacity=0.9] (12) -- (20);
				\draw [ thick, red, line width=0.15mm, opacity=0.9] (13) -- (21);
				\draw [ thick, red, line width=0.15mm, opacity=0.9] (14) -- (15);

				\draw [ thick, red, line width=0.15mm, opacity=0.9] (15) -- (22);
				\draw [ thick, red, line width=0.15mm, opacity=0.9] (16) -- (23);
				\draw [ thick, red, line width=0.15mm, opacity=0.9] (17) -- (24);
				\draw [ thick, red, line width=0.15mm, opacity=0.9] (18) -- (25);
				\draw [ thick, red, line width=0.15mm, opacity=0.9] (19) -- (26);
				\draw [ thick, red, line width=0.15mm, opacity=0.9] (20) -- (27);
				\draw [ thick, red, line width=0.15mm, opacity=0.9] (21) -- (28);
				
				\draw [ thick, red, line width=0.15mm, opacity=0.9] (15) -- (28);
				\draw [ thick, red, line width=0.15mm, opacity=0.9] (16) -- (22);
				\draw [ thick, red, line width=0.15mm, opacity=0.9] (17) -- (23);
				\draw [ thick, red, line width=0.15mm, opacity=0.9] (18) -- (24);
				\draw [ thick, red, line width=0.15mm, opacity=0.9] (19) -- (25);
				\draw [ thick, red, line width=0.15mm, opacity=0.9] (20) -- (26);
				\draw [ thick, red, line width=0.15mm, opacity=0.9] (21) -- (27);

				\draw  [ dashed, gray, line width=0.15mm, opacity=0.2] (0,0) circle (1.2cm);
				\draw  [ dashed, gray, line width=0.15mm, opacity=0.2] (0,0) circle (2cm);
				\draw  [ dashed, gray, line width=0.15mm, opacity=0.2] (0,0) circle (2.8cm);
				\draw  [ dashed, gray, line width=0.15mm, opacity=0.2] (0,0) circle (3.6cm);
				\draw  [ dashed, black, line width=0.15mm, opacity=1] (0,0) circle (4.6cm);
				
				\draw node[line width=0.05mm, circle, fill=white, minimum width=.5mm, inner sep=0mm, draw=black!80, minimum size=1mm]  (90) at ({90}:0){};
				
				\draw [ thick, red, line width=0.15mm, opacity=0.9] (1) -- (90);
				\draw [ thick, red, line width=0.15mm, opacity=0.9] (2) -- (90);
				\draw [ thick, red, line width=0.15mm, opacity=0.9] (3) -- (90);
				\draw [ thick, red, line width=0.15mm, opacity=0.9] (4) -- (90);
				\draw [ thick, red, line width=0.15mm, opacity=0.9] (5) -- (90);
				\draw [ thick, red, line width=0.15mm, opacity=0.9] (6) -- (90);
				\draw [ thick, red, line width=0.15mm, opacity=0.9] (7) -- (90);
				
				\draw [ thick, red, line width=0.15mm, opacity=0.9] (28) -- (42);
				\draw [ thick, red, line width=0.15mm, opacity=0.9] (28) -- (29);
				
				\draw [ thick, red, line width=0.15mm, opacity=0.9] (22) -- (30);
				\draw [ thick, red, line width=0.15mm, opacity=0.9] (22) -- (31);
				
				\draw [ thick, red, line width=0.15mm, opacity=0.9] (23) -- (32);
				\draw [ thick, red, line width=0.15mm, opacity=0.9] (23) -- (33);
				
				\draw [ thick, red, line width=0.15mm, opacity=0.9] (24) -- (34);
				\draw [ thick, red, line width=0.15mm, opacity=0.9] (24) -- (35);
				
				\draw [ thick, red, line width=0.15mm, opacity=0.9] (25) -- (36);
				\draw [ thick, red, line width=0.15mm, opacity=0.9] (25) -- (37);
				
				\draw [ thick, red, line width=0.15mm, opacity=0.9] (26) -- (38);
				\draw [ thick, red, line width=0.15mm, opacity=0.9] (26) -- (39);
				
				\draw [ thick, red, line width=0.15mm, opacity=0.9] (27) -- (40);
				\draw [ thick, red, line width=0.15mm, opacity=0.9] (27) -- (41);

			\end{tikzpicture}}        
	\caption{$M_{3}(C_7)$}
	\label{fig:M3-7}
\end{figure}

In \cite{VT95} (see also \cite{V87}), the graphs $M_{k}(C_{2k+1})$ are introduced as the potentially the smallest 4-chromatic graphs of odd girth at least $2k+1$. Noting that $M_{k}(C_{2k+1})$ has $2k^2+k+1$ vertices, a lower bound of $(k-1)^2$ is proved on the order of such graphs in \cite{J01}. The proof is a basic modification of \cite{N99} where a lower bound of $\frac{(k-1)^2}{2}$ was proved. For the subclass of 4-chromatic projective planar graphs of odd girth at least $2k+1$, it is proved in \cite{ES18} that $M_{k}(C_{2k+1})$ has the smallest order.

As mentioned before, the graphs constructed here can be viewed as signed graphs where all edges are negative. A bipartite analogue of these graphs, first introduced in an unpublished work \cite{N22}, are built quite similarly. We view $P_{\ell}\times C_{2k+1}$ as a signed graph where all edges are negative, then add universal vertex to the bottom layer which is connected to every second vertex of face $F_1$ by a negative edge. On the top two layer, or equivalently, face $F_{\ell}$, each vertex is connected to the antipodal vertex on the $F_{\ell}$ with a positive edge. Let $BM_{\ell, 2k+1}$ be the result signed graph and see \Cref{fig:BMl2k+1} for a depiction of this construction and \Cref{fig:BM3-7} for a specific example.

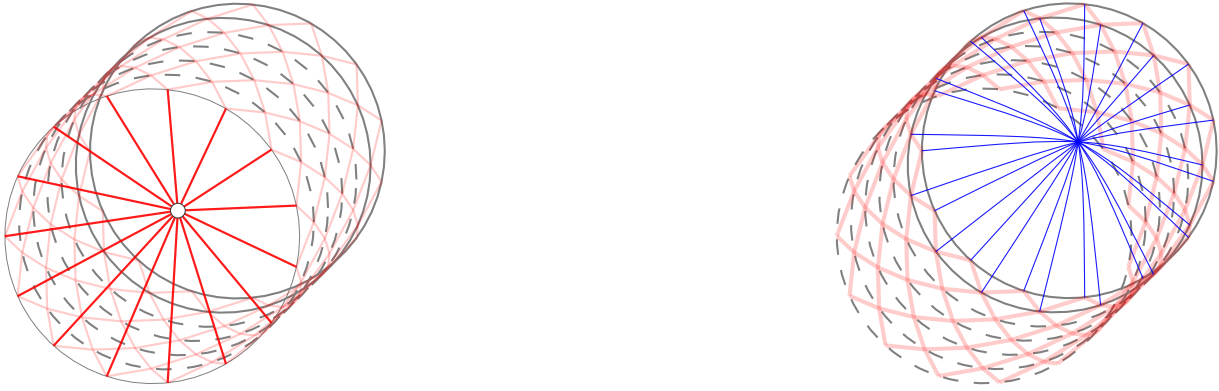
\begin{figure}[!htb]
	\centering
	\scalebox{1.5}{	\begin{minipage}{.4\textwidth}

		\scalebox{1.3}{
			\tikzset{math3d/.style= {x={(2cm,0cm)}, y={(0cm,2cm)}}}
			\begin{tikzpicture}[math3d]
				\newcommand{\n}{15}
				\newcommand{\h}{7}
				\newcommand{\ch}{12}
				\newcommand{\rl}{1}
				\newcommand{\rh}{1}
				
				\path[draw, gray, line width=0.05mm] plot[domain=0:2*pi,samples=4*\n] ({\rh*cos(\x r)}, {\rh*sin(\x r)}, 3.5);
				
				\foreach \i in{1,2}{
					\path[draw, gray] plot[domain=0:2*pi,samples=4*\n] ({\rh*cos(\x r)}, {\rh*sin(\x r)}, .5*\i);
				}
				
				\foreach \i in {3,...,6}{
					\path[draw, dashed, gray] plot[domain=0:2*pi,samples=4*\n] ({\rh*cos(\x r)}, {\rh*sin(\x r)}, .5*\i);

				}
				\foreach \i in {1,3,5}{
					\foreach \t in {1,...,\n} {
						
						\draw [thick, red, line width=0.15mm, opacity=0.2] ( ({(\rl)*cos((2*\t+1)*pi/\n r)},{\rl*sin((2*\t+1)*pi/\n r)},.5*\i) -- ({\rh*cos((2*\t+1)*pi/\n r-\ch)},{\rh*sin((2*\t+1)*pi/\n r-\ch)},.5*\i+.5)  -- cycle;
						
						\draw [ thick, red, line width=0.15mm, opacity=0.2] ( ({(\rl)*cos((2*\t+1)*pi/\n r)},{\rl*sin((2*\t+1)*pi/\n r)},.5*\i) -- ({\rh*cos((2*\t+1)*pi/\n r+\ch)},{\rh*sin((2*\t+1)*pi/\n r+\ch)},.5*\i+.5)  -- cycle;
					}

				}
				
				\foreach \i in {2,4,6}{
					\foreach \t in {1,...,\n} {
						
						\draw [ thick, red, line width=0.15mm, opacity=0.2] ( ({(\rl)*cos((2*\t+1)*pi/\n r+\ch)},{\rl*sin((2*\t+1)*pi/\n r+\ch)},.5*\i) -- ({\rh*cos((2*\t+1)*pi/\n r)},{\rh*sin((2*\t+1)*pi/\n r)},.5*\i+.5)  -- cycle;
						
						\draw [ thick, red, line width=0.15mm, opacity=0.2] ( ({(\rl)*cos((2*\t+1)*pi/\n r-\ch)},{\rl*sin((2*\t+1)*pi/\n r-\ch)},.5*\i) -- ({\rh*cos((2*\t+1)*pi/\n r)},{\rh*sin((2*\t+1)*pi/\n r)},.5*\i+.5)  -- cycle;
					}
				}

				\foreach \t in {1,...,\n} {
					
					\draw [ thick, red, line width=0.15mm, opacity=0.9] ( ({(cos((2*\t+14)*pi/\n r+\ch)},{sin((2*\t+14)*pi/\n r+\ch)},3.5) -- (-.5,-.5)  -- cycle;
					
				}
				\draw node[line width=0.05mm, circle, fill=white, minimum width=.5mm, inner sep=0mm, draw=black!80, minimum size=1mm] at (-.5,-.5) {};
				
					%
					%
				
			\end{tikzpicture}
		}
	\end{minipage}
	\begin{minipage}{.4\textwidth}
		\scalebox{1.3}{
			\tikzset{math3d/.style= {x={(2cm,0cm)}, y={(0cm,2cm)}}}
			\begin{tikzpicture}[math3d]
				\newcommand{\n}{15}
				\newcommand{\h}{7}
				\newcommand{\ch}{12}
				\newcommand{\rl}{1}
				\newcommand{\rh}{1}
				
				\foreach \i in{1,2}{
					\path[draw, gray] plot[domain=0:2*pi,samples=4*\n] ({\rh*cos(\x r)}, {\rh*sin(\x r)}, .5*\i);
				}
				
				\foreach \i in {3,...,\h}{
					\path[draw, dashed, gray] plot[domain=0:2*pi,samples=4*\n] ({\rh*cos(\x r)}, {\rh*sin(\x r)}, .5*\i);
					
				}
				\foreach \i in {1,3,5}{
					\foreach \t in {1,...,\n} {
						
						\draw [thick, red, opacity=0.2] ( ({(\rl)*cos((2*\t+1)*pi/\n r)},{\rl*sin((2*\t+1)*pi/\n r)},.5*\i) -- ({\rh*cos((2*\t+1)*pi/\n r-\ch)},{\rh*sin((2*\t+1)*pi/\n r-\ch)},.5*\i+.5)  -- cycle;
						
						\draw [ thick, red, opacity=0.2] ( ({(\rl)*cos((2*\t+1)*pi/\n r)},{\rl*sin((2*\t+1)*pi/\n r)},.5*\i) -- ({\rh*cos((2*\t+1)*pi/\n r+\ch)},{\rh*sin((2*\t+1)*pi/\n r+\ch)},.5*\i+.5)  -- cycle;
					}

				}
				
				\foreach \i in {2,4,6}{
					\foreach \t in {1,...,\n} {
						
						\draw [ thick, red, opacity=0.2] ( ({(\rl)*cos((2*\t+1)*pi/\n r+\ch)},{\rl*sin((2*\t+1)*pi/\n r+\ch)},.5*\i) -- ({\rh*cos((2*\t+1)*pi/\n r)},{\rh*sin((2*\t+1)*pi/\n r)},.5*\i+.5)  -- cycle;
						
						\draw [ thick, red, opacity=0.2] ( ({(\rl)*cos((2*\t+1)*pi/\n r-\ch)},{\rl*sin((2*\t+1)*pi/\n r-\ch)},.5*\i) -- ({\rh*cos((2*\t+1)*pi/\n r)},{\rh*sin((2*\t+1)*pi/\n r)},.5*\i+.5)  -- cycle;
					}
				}

				\foreach \t in {1,...,\n} {
					
					\draw [line width=0.075mm,  blue, opacity=0.9] ( ({(cos((2*\t+14)*pi/\n r+\ch)},{sin((2*\t+14)*pi/\n r+\ch)},.5) .. controls ({cos((2*\t+14)*pi/\n r+\ch)+cos((2*\t)*pi/\n r)}, {sin((2*\t+14)*pi/\n r+\ch)+sin((2*\t)*pi/\n r)})  .. ({cos((2*\t)*pi/\n r)},{sin((2*\t)*pi/\n r)},1);
					
				}

			\end{tikzpicture}
		}
	\end{minipage}}        
	\caption{Bipartite analogue}
	\label{fig:BMl2k+1}
\end{figure}

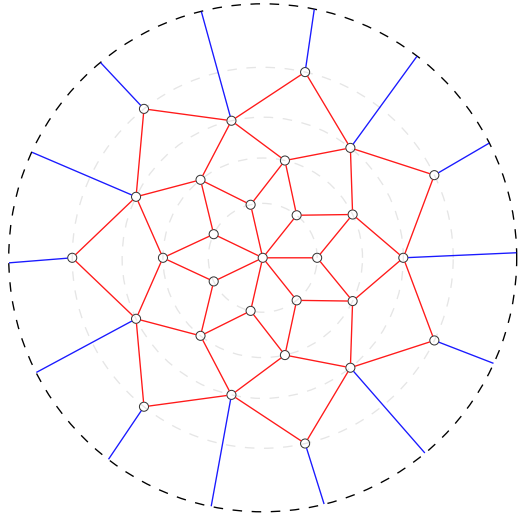
\begin{figure}[!htb]
	\centering
	\scalebox{1.2}{	\begin{minipage}[t]{.4\textwidth}
		\centering
		\scalebox{1}{
			\begin{tikzpicture}
				
				\foreach \i in {1,2,3,4,5,6,7}
				{ \draw node[line width=0.05mm, circle, fill=white, minimum width=.5mm, inner sep=0mm, draw=black!80, minimum size=1mm]  (\i) at ({\i*360/7}:1.2){};}
				
				\foreach \i in {8,9,10,11,12,13,14}
				{ \draw node[line width=0.05mm, circle, fill=white, minimum width=.5mm, inner sep=0mm, draw=black!80, minimum size=1mm]  (\i) at ({\i*360/7+180/7}:2.2){};
				}
				
				\foreach \i in {15,16,17,18,19,20,21}
				{ \draw node[line width=0.05mm, circle, fill=white, minimum width=.5mm, inner sep=0mm, draw=black!80, minimum size=1mm]  (\i) at ({\i*360/7}:3.1){};}
				
				\foreach \i in {22,23,24,25,26,27,28}
				{ \draw node[line width=0.05mm, circle, fill=white, minimum width=.5mm, inner sep=0mm, draw=black!80, minimum size=1mm]  (\i) at ({\i*360/7+180/7}:4.2){};}
				
				\foreach \i in {29,30,...,42}
				{ \draw node[line width=0.05mm, circle, fill=white, minimum width=.5mm, inner sep=0mm, draw=black!80, minimum size=0mm]  (\i) at ({\i*360/14+180/14+14}:5.6){};}

				\draw [ thick, red, line width=0.15mm, opacity=0.9] (1) -- (8);
				\draw [ thick, red, line width=0.15mm, opacity=0.9] (2) -- (9);
				\draw [ thick, red, line width=0.15mm, opacity=0.9] (3) -- (10);
				\draw [ thick, red, line width=0.15mm, opacity=0.9] (4) -- (11);
				\draw [ thick, red, line width=0.15mm, opacity=0.9] (5) -- (12);
				\draw [ thick, red, line width=0.15mm, opacity=0.9] (6) -- (13);
				\draw [ thick, red, line width=0.15mm, opacity=0.9] (7) -- (14);
				
				\draw [ thick, red, line width=0.15mm, opacity=0.9] (1) -- (14);
				\draw [ thick, red, line width=0.15mm, opacity=0.9] (2) -- (8);
				\draw [ thick, red, line width=0.15mm, opacity=0.9] (3) -- (9);
				\draw [ thick, red, line width=0.15mm, opacity=0.9] (4) -- (10);
				\draw [ thick, red, line width=0.15mm, opacity=0.9] (5) -- (11);
				\draw [ thick, red, line width=0.15mm, opacity=0.9] (6) -- (12);
				\draw [ thick, red, line width=0.15mm, opacity=0.9] (7) -- (13);

				\draw [ thick, red, line width=0.15mm, opacity=0.9] (8) -- (15);
				\draw [ thick, red, line width=0.15mm, opacity=0.9] (9) -- (16);
				\draw [ thick, red, line width=0.15mm, opacity=0.9] (10) -- (17);
				\draw [ thick, red, line width=0.15mm, opacity=0.9] (11) -- (18);
				\draw [ thick, red, line width=0.15mm, opacity=0.9] (12) -- (19);
				\draw [ thick, red, line width=0.15mm, opacity=0.9] (13) -- (20);
				\draw [ thick, red, line width=0.15mm, opacity=0.9] (14) -- (21);

				\draw [ thick, red, line width=0.15mm, opacity=0.9] (8) -- (16);
				\draw [ thick, red, line width=0.15mm, opacity=0.9] (9) -- (17);
				\draw [ thick, red, line width=0.15mm, opacity=0.9] (10) -- (18);
				\draw [ thick, red, line width=0.15mm, opacity=0.9] (11) -- (19);
				\draw [ thick, red, line width=0.15mm, opacity=0.9] (12) -- (20);
				\draw [ thick, red, line width=0.15mm, opacity=0.9] (13) -- (21);
				\draw [ thick, red, line width=0.15mm, opacity=0.9] (14) -- (15);

				\draw [ thick, red, line width=0.15mm, opacity=0.9] (15) -- (22);
				\draw [ thick, red, line width=0.15mm, opacity=0.9] (16) -- (23);
				\draw [ thick, red, line width=0.15mm, opacity=0.9] (17) -- (24);
				\draw [ thick, red, line width=0.15mm, opacity=0.9] (18) -- (25);
				\draw [ thick, red, line width=0.15mm, opacity=0.9] (19) -- (26);
				\draw [ thick, red, line width=0.15mm, opacity=0.9] (20) -- (27);
				\draw [ thick, red, line width=0.15mm, opacity=0.9] (21) -- (28);
				
				\draw [ thick, red, line width=0.15mm, opacity=0.9] (15) -- (28);
				\draw [ thick, red, line width=0.15mm, opacity=0.9] (16) -- (22);
				\draw [ thick, red, line width=0.15mm, opacity=0.9] (17) -- (23);
				\draw [ thick, red, line width=0.15mm, opacity=0.9] (18) -- (24);
				\draw [ thick, red, line width=0.15mm, opacity=0.9] (19) -- (25);
				\draw [ thick, red, line width=0.15mm, opacity=0.9] (20) -- (26);
				\draw [ thick, red, line width=0.15mm, opacity=0.9] (21) -- (27);

				\draw  [ dashed, gray, line width=0.15mm, opacity=0.2] (0,0) circle (1.2cm);
				\draw  [ dashed, gray, line width=0.15mm, opacity=0.2] (0,0) circle (2.2cm);
				\draw  [ dashed, gray, line width=0.15mm, opacity=0.2] (0,0) circle (3.1cm);
				\draw  [ dashed, gray, line width=0.15mm, opacity=0.2] (0,0) circle (4.2cm);
				\draw  [ dashed, black, line width=0.15mm, opacity=1] (0,0) circle (5.6cm);
				
				\draw node[line width=0.05mm, circle, fill=white, minimum width=.5mm, inner sep=0mm, draw=black!80, minimum size=1mm]  (90) at ({90}:0){};
				
				\draw [ thick, red, line width=0.15mm, opacity=0.9] (1) -- (90);
				\draw [ thick, red, line width=0.15mm, opacity=0.9] (2) -- (90);
				\draw [ thick, red, line width=0.15mm, opacity=0.9] (3) -- (90);
				\draw [ thick, red, line width=0.15mm, opacity=0.9] (4) -- (90);
				\draw [ thick, red, line width=0.15mm, opacity=0.9] (5) -- (90);
				\draw [ thick, red, line width=0.15mm, opacity=0.9] (6) -- (90);
				\draw [ thick, red, line width=0.15mm, opacity=0.9] (7) -- (90);
				
				\draw [ thick, blue, line width=0.15mm, opacity=0.9] (28) -- (42);
				\draw [ thick, blue, line width=0.15mm, opacity=0.9] (15) -- (29);
				
				\draw [ thick, blue, line width=0.15mm, opacity=0.9] (22) -- (30);
				\draw [ thick, blue, line width=0.15mm, opacity=0.9] (16) -- (31);
				
				\draw [ thick, blue, line width=0.15mm, opacity=0.9] (23) -- (32);
				\draw [ thick, blue, line width=0.15mm, opacity=0.9] (17) -- (33);
				
				\draw [ thick, blue, line width=0.15mm, opacity=0.9] (24) -- (34);
				\draw [ thick, blue, line width=0.15mm, opacity=0.9] (18) -- (35);
				
				\draw [ thick, blue, line width=0.15mm, opacity=0.9] (25) -- (36);
				\draw [ thick, blue, line width=0.15mm, opacity=0.9] (19) -- (37);
				
				\draw [ thick, blue, line width=0.15mm, opacity=0.9] (26) -- (38);
				\draw [ thick, blue, line width=0.15mm, opacity=0.9] (20) -- (39);
				
				\draw [ thick, blue, line width=0.15mm, opacity=0.9] (27) -- (40);
				\draw [ thick, blue, line width=0.15mm, opacity=0.9] (21) -- (41);

			\end{tikzpicture}
		}

	\end{minipage}}        
	\caption{$\widehat{BM_{3}(C_7)}$}
	\label{fig:BM3-7}
\end{figure}

Some key properties of $BM_{\ell, 2k+1}$ are stated in the following proposition. We refer to \cite{N22} for details on the first two claims, noting that for the last item one may apply \Cref{cor:SignedQuad}.

\begin{proposition}
	The signed graph $BM_{\ell, 2k+1}$ satisfies the followings.
	\begin{itemize}
		\item Its underlying graph is bipartite.
		\item Its negative girth is $\min \{2l, 2k+2\}$.
		\item It admits an embedding on the projective plane where all faces are positive 4-cycles.
		\item Its circular chromatic number is 4.
	\end{itemize}
\end{proposition}

For the purpose of studying the circular chromatic number of signed graphs we may assume each vertex has a positive loop on it. With such assumption, one observes that identifying thtwo ends of each positive edge of $BM_{\ell, 2k+1}$ is, simultaneously, both a minor and a homomorphism operation result of which is $M_{\ell-1}(C_{2k+1})$. Thus,providing lower bounds o parameters such as the circular chromatic number of these signed bipartite graphs also verifies the claims on classic graph counterpart. 

Among these graphs the smallest having negative girth $2k$ is $BM_{k, 2k-1}$ which has $2k^2-k+1$ vertices. Perhaps this is the best possible order of a signed bipartite graph of negative girth $2k$ having circular chromatic number 4. However, in the following we introduce another family of the same order.

\begin{figure}[!htb]
	\centering
	\scalebox{.8}{			\begin{tikzpicture}
				
			\draw node[line width=0.05mm, circle, fill=white, minimum width=.5mm, inner sep=0mm, draw=black!80, minimum size=1mm]  (90) at ({90}:0){};	
				
			\foreach \i in {1,...,8}
			{ \draw node[line width=0.05mm, circle, fill=white, minimum width=.5mm, inner sep=0mm, draw=black!80, minimum size=1mm]  (\i) at ({\i*360/8}:1.2){};}
			
			\foreach \i in {9,...,16}
			{ \draw node[line width=0.05mm, circle, fill=white, minimum width=.5mm, inner sep=0mm, draw=black!80, minimum size=1mm]  (\i) at ({\i*360/8+180/8}:3){};}
			
			\foreach \i in {17,...,24}
			{ \draw node[line width=0.05mm, circle, fill=white, minimum width=.5mm, inner sep=0mm, draw=black!80, minimum size=1mm]  (\i) at ({\i*360/8}:5){};}
			
			\foreach \i in {25,...,32}
			{ \draw node[line width=0.05mm, circle, fill=white, minimum width=.5mm, inner sep=0mm, draw=black!80, minimum size=1mm]  (\i) at ({\i*360/8+180/8}:7){};}
			
			\foreach \i in {33,...,48}
			{ \draw node[line width=0.05mm, circle, fill=white, minimum width=.5mm, inner sep=0mm, draw=black!80, minimum size=0mm]  (\i) at ({mod(\i,16)*360/16}:9){};}
			
			\foreach \i / \j in {1/90,2/90, 3/90,4/90,5/90,6/90,7/90,8/90,
				                 1/9,2/10,3/11,4/12,5/13,6/14, 7/15,8/16, 
			                    	9/2,10/3,11/4,12/5,13/6,14/7,15/8,16/1,
			                    	9/17,10/18,11/19,12/20, 13/21,14/22, 15/23,16/24,
			                    	18/9,19/10,20/11,21/12,22/13,23/14,24/15,17/16,
			                    	17/25,18/26,19/27,20/28,21/29,22/30,23/31,24/32,
			                    	25/18,26/19,27/20,28/21,29/22,30/23,31/24,32/17}{
			\draw [ thick, red, line width=0.15mm, opacity=0.9] (\i) -- (\j);}

				\foreach \i / \j in {17/34, 18/36, 19/38, 20/40, 21/42, 22/44, 23/46, 24/48, 
				                     25/35, 26/37, 27/39, 28/41, 29/43, 30/45, 31/47, 32/33}{
				\draw [ thick, red, line width=0.15mm, opacity=0.9] (\i) -- (\j);}

			\draw  [ dashed, gray, line width=0.15mm, opacity=0.2] (0,0) circle (1.2cm);
			\draw  [ dashed, gray, line width=0.15mm, opacity=0.2] (0,0) circle (3cm);
			\draw  [ dashed, gray, line width=0.15mm, opacity=0.2] (0,0) circle (5cm);
			\draw  [ dashed, gray, line width=0.15mm, opacity=0.2] (0,0) circle (7cm);
			\draw  [ dashed, black, line width=0.15mm, opacity=1] (0,0) circle (9cm);

			\end{tikzpicture}}        
	\caption{$M_{3}(C_8)$}
	\label{fig:M3-8}
\end{figure}

Our next construction is based on $P_{\ell}\times C_{2k}$ to which we add a new vertex, say $u$, joined to all $2k$ vertices of the first layer. Viewing the graph on the plane with a natural embedding, where $u$ is in the center, the outer face is a $4k$-cycle. Connecting antipodal vertices of this cycle through a cross cap we get a quadrangulation of the projective plane whose odd girth is the $\min \{2l+1, 2k+1 \}$, as it is quadrangulation of the projective plane which is not bipartite, it has circular chromatic number $4$. We name this graph $M_{\ell}(C_{2k})$, in \Cref{fig:M3-8} a presentation of $M_{3}(C_{8})$ is given. If for each of the edge going through the cross-cap we switch at one end and contract it, and then, remove all vertices of degree 2, the result, denoted $BM_{\ell}(C_{2k})$, is a signed bipartite graph embedded on the projective where all faces are positive $4$-cycles and negative girth is $\min \{2l, 2k \}$. A general presentation of the these signed graphs is given in \Cref{fig:BMl-2k} and a specific example is in \Cref{fig:BM4-8}. The signed bipartite graph $BM_{k}(C_{2k})$ is of negative girth $2k$ and is on $2k^2-k+1$ vertices.    

\begin{figure}[!htb]
	\centering

			\begin{minipage}{.4\textwidth}
			
			\scalebox{1.3}{
				\tikzset{math3d/.style= {x={(2cm,0cm)}, y={(0cm,2cm)}}}
				\begin{tikzpicture}[math3d]
					\newcommand{\n}{16}
					\newcommand{\h}{7}
					\newcommand{\ch}{12}
					\newcommand{\rl}{1}
					\newcommand{\rh}{1}

					\path[draw, gray] plot[domain=0:2*pi,samples=4*\n] ({\rh*cos(\x r)}, {\rh*sin(\x r)}, 3.5);
					
					\foreach \i in {1,...,6}{
						\path[draw, dashed, gray] plot[domain=0:2*pi,samples=4*\n] ({\rh*cos(\x r)}, {\rh*sin(\x r)}, .5*\i);
					}

					\foreach \i in {1,3,5}{
						\foreach \t in {1,...,\n} {
							
							\draw [thick, red, opacity=0.2] ( ({(\rl)*cos((2*\t+1)*pi/\n r)},{\rl*sin((2*\t+1)*pi/\n r)},.5*\i) -- ({\rh*cos((2*\t+1)*pi/\n r-\ch)},{\rh*sin((2*\t+1)*pi/\n r-\ch)},.5*\i+.5)  -- cycle;
							
							\draw [ thick, red, opacity=0.2] ( ({(\rl)*cos((2*\t+1)*pi/\n r)},{\rl*sin((2*\t+1)*pi/\n r)},.5*\i) -- ({\rh*cos((2*\t+1)*pi/\n r+\ch)},{\rh*sin((2*\t+1)*pi/\n r+\ch)},.5*\i+.5)  -- cycle;
						}

					}
					
					\foreach \i in {2,4,6}{
						\foreach \t in {1,...,\n} {
							
							\draw [ thick, red, opacity=0.2] ( ({(\rl)*cos((2*\t+1)*pi/\n r+\ch)},{\rl*sin((2*\t+1)*pi/\n r+\ch)},.5*\i) -- ({\rh*cos((2*\t+1)*pi/\n r)},{\rh*sin((2*\t+1)*pi/\n r)},.5*\i+.5)  -- cycle;
							
							\draw [ thick, red, opacity=0.2] ( ({(\rl)*cos((2*\t+1)*pi/\n r-\ch)},{\rl*sin((2*\t+1)*pi/\n r-\ch)},.5*\i) -- ({\rh*cos((2*\t+1)*pi/\n r)},{\rh*sin((2*\t+1)*pi/\n r)},.5*\i+.5)  -- cycle;
						}
					}

					\foreach \t in {1,...,\n} {
						
						\draw [  red, opacity=1] ( ({(cos((2*\t+14)*pi/\n r+\ch)},{sin((2*\t+14)*pi/\n r+\ch)},3.5) -- (-.5,-.5)  -- cycle;
						
					}
					\draw node[circle, red, inner sep=0mm, draw=black!80, fill=white, minimum size=1mm] at (-.5,-.5) {};
				\end{tikzpicture}
			}
		\end{minipage}
		\begin{minipage}{.4\textwidth}
			\scalebox{1.3}{
				\tikzset{math3d/.style= {x={(2cm,0cm)}, y={(0cm,2cm)}}}
				\begin{tikzpicture}[math3d]
					\newcommand{\n}{16}
					\newcommand{\h}{7}
					\newcommand{\hl}{8}
					\newcommand{\ch}{12}
					\newcommand{\rl}{1}
					\newcommand{\rh}{1}
					
					
					\foreach \i in {1,...,\h}{
						\path[draw, dashed, gray] plot[domain=0:2*pi,samples=4*\n] ({\rh*cos(\x r)}, {\rh*sin(\x r)}, .5*\i);

					}
					\foreach \i in {1,3,5}{
						\foreach \t in {1,...,\n} {
							
							\draw [thick, red, opacity=0.2] ( ({(\rl)*cos((2*\t+1)*pi/\n r)},{\rl*sin((2*\t+1)*pi/\n r)},.5*\i) -- ({\rh*cos((2*\t+1)*pi/\n r-\ch)},{\rh*sin((2*\t+1)*pi/\n r-\ch)},.5*\i+.5)  -- cycle;
							
							\draw [ thick, red, opacity=0.2] ( ({(\rl)*cos((2*\t+1)*pi/\n r)},{\rl*sin((2*\t+1)*pi/\n r)},.5*\i) -- ({\rh*cos((2*\t+1)*pi/\n r+\ch)},{\rh*sin((2*\t+1)*pi/\n r+\ch)},.5*\i+.5)  -- cycle;
						}

					}
					
					\foreach \i in {2,4,6}{
						\foreach \t in {1,...,\n} {
							
							\draw [ thick, red, opacity=0.2] ( ({(\rl)*cos((2*\t+1)*pi/\n r+\ch)},{\rl*sin((2*\t+1)*pi/\n r+\ch)},.5*\i) -- ({\rh*cos((2*\t+1)*pi/\n r)},{\rh*sin((2*\t+1)*pi/\n r)},.5*\i+.5)  -- cycle;
							
							\draw [ thick, red, opacity=0.2] ( ({(\rl)*cos((2*\t+1)*pi/\n r-\ch)},{\rl*sin((2*\t+1)*pi/\n r-\ch)},.5*\i) -- ({\rh*cos((2*\t+1)*pi/\n r)},{\rh*sin((2*\t+1)*pi/\n r)},.5*\i+.5)  -- cycle;
						}
					}

						%
						%
					
					\foreach \t in {1,...,\n} {
						
						\draw node[circle, red, inner sep=0mm, draw=black!80, fill=white, minimum size=.5mm] at (({cos((2*\t+14)*pi/\n r+\ch)},{sin((2*\t+14)*pi/\n r+\ch)},.5)  (x\t) {};

					}
					
					\foreach \t in {1,...,\hl} {
						
						\draw node[circle, red, inner sep=0mm, draw=black!80, fill=white, minimum size=1mm] at (({cos((1.725*\t+12)*pi/\n r+\ch)/1.6},{sin((1.725*\t+12)*pi/\n r+\ch)/1.8})  (y\t){};

					}
					
					\foreach \i/\j in { 15/1, 16/1,16/2, 1/2,1/3,2/3,2/4, 3/4, 3/5, 4/5, 4/6, 5/6, 5/7, 6/7, 6/8, 7/8} {
						
						\draw[line width=0.2mm, red ] (x\i) -- (y\j) ;}
					
					\foreach \i/\j in { 15/8, 14/8, 14/7, 13/7, 13/6, 12/6, 12/5, 11/5, 11/4, 10/4, 10/3, 9/3, 9/2, 8/2, 8/1, 7/1} {
						
						\draw[line width=0.2mm,blue] (x\i) -- (y\j) ;}
					
				\end{tikzpicture}
			}
		\end{minipage}  
	\caption{$\widehat{BM_{\ell}(C_{2k})}$}
	\label{fig:BMl-2k}
\end{figure}
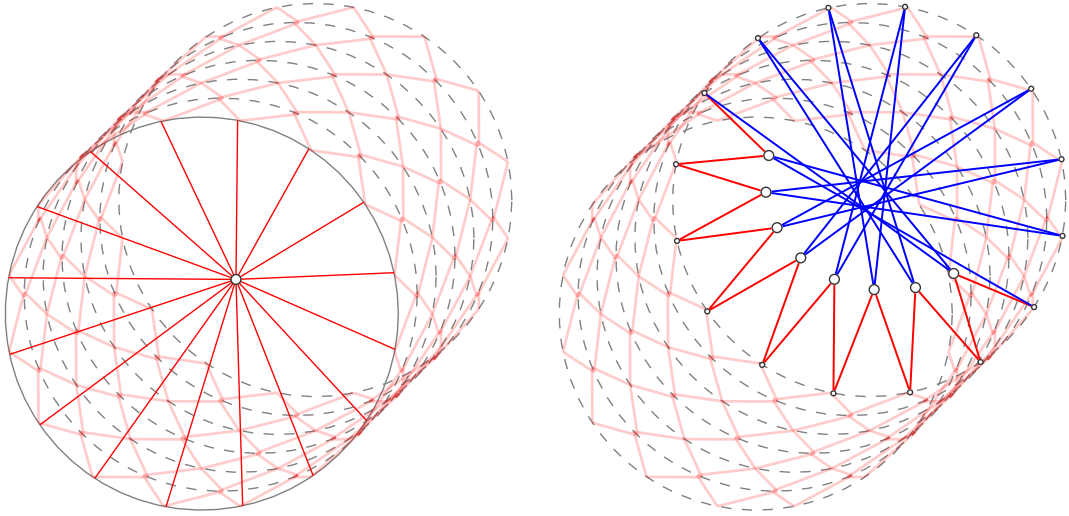

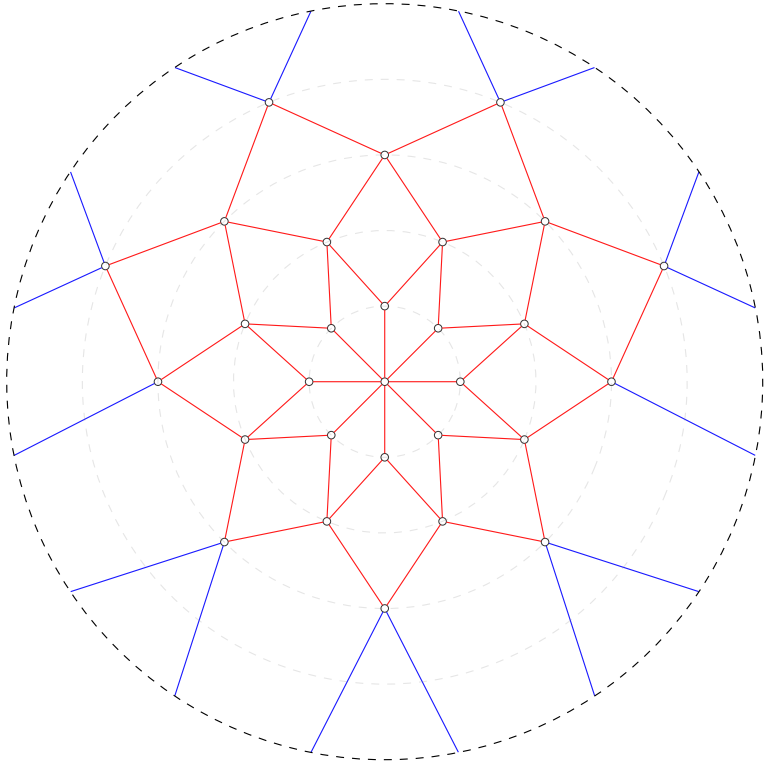
\begin{figure}[!htb]
	\centering
	\scalebox{1}{			\begin{tikzpicture}[rotate=-22.5]

%
%
%
%
%
%

			\draw node[line width=0.05mm, circle, fill=white, minimum width=.5mm, inner sep=0mm, draw=black!80, minimum size=1mm]  (90) at ({90}:0){};	
			
			
			\foreach \i in {9,...,16}
			{ \draw node[line width=0.05mm, circle, fill=white, minimum width=.5mm, inner sep=0mm, draw=black!80, minimum size=1mm]  (\i) at ({\i*360/8+180/8}:2.0){};}
			
			\foreach \i in {17,...,24}
			{ \draw node[line width=0.05mm, circle, fill=white, minimum width=.5mm, inner sep=0mm, draw=black!80, minimum size=1mm]  (\i) at ({\i*360/8}:4.0){};}
			
			\foreach \i in {25,...,32}
			{ \draw node[line width=0.05mm, circle, fill=white, minimum width=.5mm, inner sep=0mm, draw=black!80, minimum size=1mm]  (\i) at ({\i*360/8+180/8}:6.0){};}
			
			\foreach \i in {33,...,36}
			{ \draw node[line width=0.05mm, circle, fill=white, minimum width=.5mm, inner sep=0mm, draw=black!80, minimum size=1mm]  (\i) at ({mod(\i,8)*360/8}:8.0){};}
			
			\foreach \i in {41,...,56}
			{ \draw node[line width=0.05mm, circle, fill=white, minimum width=.5mm, inner sep=0mm, draw=black!80, minimum size=0mm]  (\i) at ({mod(\i,16)*360/16+180/16}:10.0){};}
			
			\foreach \i / \j in {9/90,10/90, 11/90,12/90,13/90,14/90,15/90,16/90,
				9/17,10/18,11/19,12/20, 13/21,14/22, 15/23,16/24,
				18/9,19/10,20/11,21/12,22/13,23/14,24/15,17/16,
				17/25,18/26,19/27,20/28,21/29,22/30,23/31,24/32,
				25/18,26/19,27/20,28/21,29/22,30/23,31/24,32/17, 
				25/33,26/34,27/35,28/36,
				33/32,34/25,35/26,36/27}{
				\draw [ thick, red, line width=0.15mm, opacity=0.9] (\i) -- (\j);}

			\foreach \i / \j in {33/50,34/52,35/54,36/56, 33/49, 34/51,35/53, 36/55,
			28/41, 29/43, 30/45, 31/47,
			29/42,30/44,31/46,32/48}{
				\draw [ thick, blue, line width=0.15mm, opacity=0.9] (\i) -- (\j);}

			\draw  [ dashed, gray, line width=0.15mm, opacity=0.2] (0,0) circle (2.0cm);
			\draw  [ dashed, gray, line width=0.15mm, opacity=0.2] (0,0) circle (4.0cm);
			\draw  [ dashed, gray, line width=0.15mm, opacity=0.2] (0,0) circle (6.0cm);
			\draw  [ dashed, gray, line width=0.15mm, opacity=0.2] (0,0) circle (8.0cm);
			\draw  [ dashed, black, line width=0.15mm, opacity=1] (0,0) circle (10.0cm);

			\end{tikzpicture}}        
	\caption{$\widehat{BM_{4}(C_8)}$}
	\label{fig:BM4-8}
\end{figure}

\subsection{Positive triangulations of the projective plane}

To build example of signed projective planar graphs where all faces are positive triangles, we may take any signed qudrangulation where each face is positive 4-cycle and inside each face insert an edge connecting one of the diagonal pairs. Then we assign one of the two possible signs to the new edge such that the two resulting triangles are positive. Having done this for all the faces, we will have a triangulation of the projective plane where all the faces are positive triangles. If the quderangulation we started with is not balanced, then, by \Cref{cor:triangulation}, the resulting triangulation would be of circular chromatic number 6.
The choices made in connecting one of the two diagonal pairs of each face may overall impact the negative girth of the final result. Nevertheless, the negative girth of the result would never be less than half the negative girth of the quadrangulation we have started with.  In practice,  a better choice of connecting diagonal faces may result in a positive triangulation whose negative girth is significantly larger that the half of the negative girth of the quadrangulation
we have started with. \Cref{fig:BM4-7-tri} present an example of better choice of diagonal edges.
In \Cref{fig:girth3,fig:girth4,fig:girth5} we have best possible examples of girth 3, 4, and 5. 	

\begin{figure}[!htb]
	\centering
	\scalebox{2}{\input{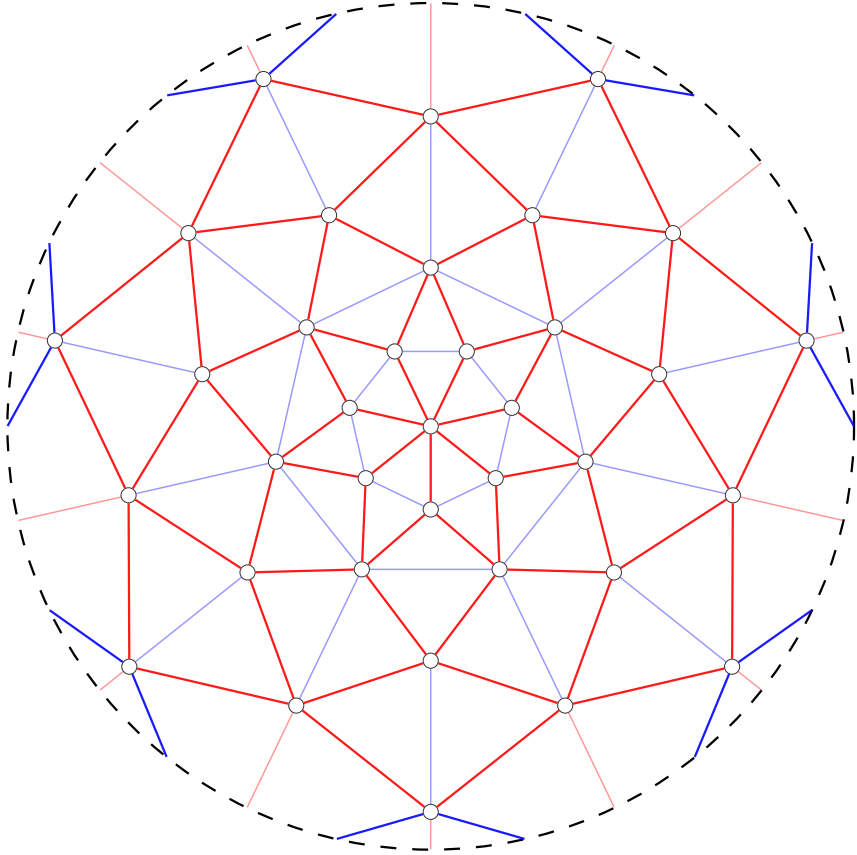}}        
	\caption{A triangulation based on $BM_4(C_7)$ whose negative girth is 7.}
	\label{fig:BM4-7-tri}
\end{figure}

In conclusion, for every integer $k$, there is a positive triangulation of the projective plane whose negative girth is at least $k$. Thus in particular we have infinitely many inclusion-wise minimal signed simple graphs embedded on the projective plane whose circular chromatic number is 6. This is in contrast with the graph case where $K_6$ is the only 6-chromatic graphs on the projective plane, and also in contrast with the result of Thomassen \cite{T97} who showed that the number minimal 6-chromatic graphs on any surface is finite. 

A lower bound on the number of vertices of a 6-chromatic signed graph of negative girth at least $k$ is implied in \cite{W25}. Further improvement together with a proof that some of the construction presented here are critical, meaning their circular chromatic number drops to 4 by removal of any edge, would be addressed in a forthcoming work.

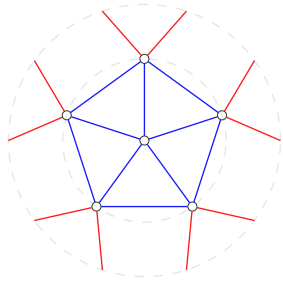
\begin{figure}[ht]
	\centering
	\scalebox{1}{	\begin{minipage}[t]{1\textwidth}
		\centering
		\scalebox{1.2}{
			\begin{tikzpicture}
				
			\foreach \i in {1,2,3,4,5}
			{ \draw node[line width=0.05mm, circle, fill=white, minimum width=1mm, inner sep=0mm, draw=black!80, minimum size=1mm]  (\i) at ({(\i-1)*72+18}:1.8){};}
			
			\draw node[line width=0.05mm, circle, fill=white, minimum width=1mm, inner sep=0mm, draw=black!80, minimum size=1mm]  (0) at ({0}:0){};
			
			\draw  [ dashed, gray, line width=0.15mm, opacity=0.2] (0,0) circle (1.8cm);
			
			\draw  [ dashed, gray, line width=0.15mm, opacity=0.2] (0,0) circle (3cm);

			\foreach \i in {1,2,3,4,5}
			{\draw [ thick, blue, line width=0.15mm, opacity=0.9] (0) -- (\i);}
			
			\foreach \i\j in {1/2,2/3,3/4,4/5,5/1}
			{\draw [ thick, blue, line width=0.15mm, opacity=0.9] (\i) -- (\j) ;}
			
			\foreach \i in {1,2,3,4,5}
			{\draw [ thick, red, line width=0.15mm, opacity=0.9] ({(\i-1)*72}:3)--(\i)--({(\i-1)*72+36}:3);}
				
			\end{tikzpicture}
		}

	\end{minipage}} 
	\caption{A positive triangulation with negative girth 3}
	\label{fig:girth3}	
\end{figure}

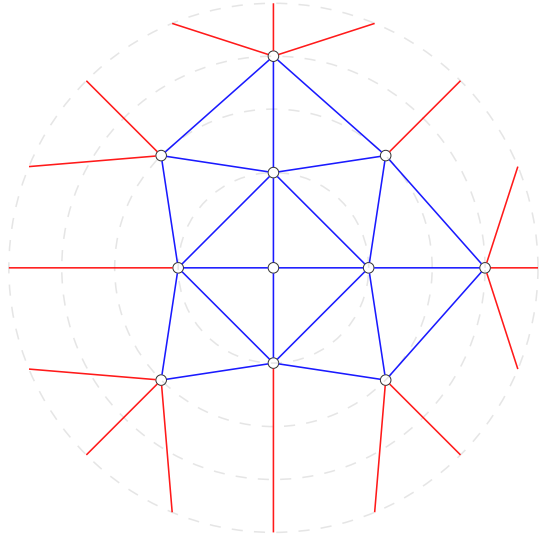
\begin{figure}[ht]
	\centering
	\scalebox{1}{	\begin{minipage}[t]{1\textwidth}
		\centering
		\scalebox{1.4}{
			\begin{tikzpicture}
				
			\foreach \i in {1,2,3,4}
			{ \draw node[line width=0.05mm, circle, fill=white, minimum width=.5mm, inner sep=0mm, draw=black!80, minimum size=1mm]  (\i) at ({(\i-1)*90}:1.8){};}
			
			\foreach \i in {5,6,7,8}
			{ \draw node[line width=0.05mm, circle, fill=white, minimum width=.5mm, inner sep=0mm, draw=black!80, minimum size=1mm]  (\i) at ({(\i-5)*90+45}:3){};
			}
			
			\foreach \i in {9,10}
			{ \draw node[line width=0.05mm, circle, fill=white, minimum width=.5mm, inner sep=0mm, draw=black!80, minimum size=1mm]  (\i) at ({(\i-9)*90}:4){};
			}

			\draw node[line width=0.05mm, circle, fill=white, minimum width=.5mm, inner sep=0mm, draw=black!80, minimum size=1mm]  (0) at ({0}:0){};
			
			\draw  [ dashed, gray, line width=0.15mm, opacity=0.2] (0,0) circle (1.8cm);
			\draw  [ dashed, gray, line width=0.15mm, opacity=0.2] (0,0) circle (3cm);
			\draw  [ dashed, gray, line width=0.15mm, opacity=0.2] (0,0) circle (4cm);
			\draw  [ dashed, gray, line width=0.15mm, opacity=0.2] (0,0) circle (5cm);

			\foreach \i in {1,2,3,4}
			{\draw [ thick, blue, line width=0.15mm, opacity=0.9] (0) -- (\i);}
			
			\foreach \i\j in {1/2,2/3,3/4,4/1}
			{\draw [ thick, blue, line width=0.15mm, opacity=0.9] (\i) -- (\j) ;}
			
			\foreach \i\j in {5/1,5/2,6/2,6/3,7/3,7/4,8/4,8/1}
			{\draw [ thick, blue, line width=0.15mm, opacity=0.9] (\i) -- (\j) ;}
			
			\foreach \i\j in {9/1,9/5,9/8,10/2,10/5,10/6}
			{\draw [ thick, blue, line width=0.15mm, opacity=0.9] (\i) -- (\j) ;}

			\foreach \i in {1,2,3}
			{	\draw [ thick, red, line width=0.15mm, opacity=0.9] 
				(9)--({(\i-2)*22.5}:5);
				\draw [ thick, red, line width=0.15mm, opacity=0.9] 
				(10)--({(\i-2)*22.5+90}:5);
				\draw [ thick, red, line width=0.15mm, opacity=0.9] 
				(7)--({(\i-2)*22.5+225}:5);
			}

			\draw [ thick, red, line width=0.15mm, opacity=0.9] 
			(5)--(45:5);
			\draw [ thick, red, line width=0.15mm, opacity=0.9] 
			(3)--(180:5);
			\draw [ thick, red, line width=0.15mm, opacity=0.9] 
			(4)--(270:5);
			\draw [ thick, red, line width=0.15mm, opacity=0.9] 
			(6)--(135:5);
			\draw [ thick, red, line width=0.15mm, opacity=0.9] 
			(8)--(315:5);
			
			\draw [ thick, red, line width=0.15mm, opacity=0.9] 
			(6)--(157.5:5);
			\draw [ thick, red, line width=0.15mm, opacity=0.9] 
			(8)--(292.5:5);
			\end{tikzpicture}
		}

	\end{minipage}} 
	\caption{A positive triangulation of negative girth 4}
	\label{fig:girth4}	
\end{figure}

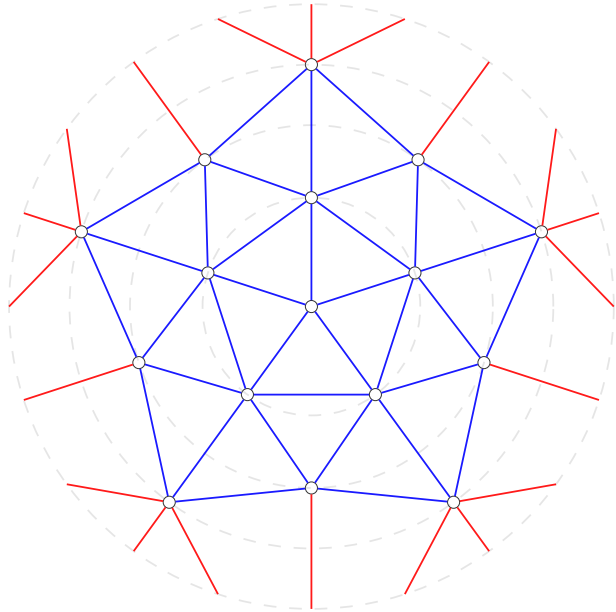
\begin{figure}[ht]
	\centering
    \scalebox{1}{	\begin{minipage}[t]{1\textwidth}
		\centering
		\scalebox{1.6}{
			\begin{tikzpicture}
				\foreach \i in {1,2,3,4,5}
				{ \draw node[line width=0.05mm, circle, fill=white, minimum width=1mm, inner sep=0mm, draw=black!80, minimum size=1mm]  (\i) at ({(\i-1)*72+90}:1.8){};}

				\foreach \i in {6,7,8,9,10}
				{ \draw node[line width=0.05mm, circle, fill=white, minimum width=1mm, inner sep=0mm, draw=black!80, minimum size=1mm]  (\i) at ({(\i-1)*72+54}:3){};}

				\foreach \i in {11,12,13,14,15}
				{ \draw node[line width=0.05mm, circle, fill=white, minimum width=1mm, inner sep=0mm, draw=black!80, minimum size=1mm]  (\i) at ({(\i-1)*72+90}:4){};}

				\draw node[line width=0.05mm, circle, fill=white, minimum width=1mm, inner sep=0mm, draw=black!80, minimum size=1mm]  (0) at ({0}:0){};
				\draw  [ dashed, gray, line width=0.15mm, opacity=0.2] (0,0) circle (5cm);
				\draw  [ dashed, gray, line width=0.15mm, opacity=0.2] (0,0) circle (4cm);
				
				\draw  [ dashed, gray, line width=0.15mm, opacity=0.2] (0,0) circle (3cm);
				
				\draw  [ dashed, gray, line width=0.15mm, opacity=0.2] (0,0) circle (1.8cm);

				\foreach \i in {1,2,3,4,5}
				{\draw [ thick, blue, line width=0.15mm, opacity=0.9] (0) -- (\i);}
				
				\foreach \i\j in {1/2,2/3,3/4,4/5,5/1}
				{\draw [ thick, blue, line width=0.15mm, opacity=0.9] (\i) -- (\j) ;}
				
				\foreach \i\j in {6/1,6/5,7/1,7/2,8/2,8/3,9/3,9/4,10/4,10/5}
				{\draw [ thick, blue, line width=0.15mm, opacity=0.9] (\i) -- (\j) ;}
				
				\foreach \i\j in {11/1,11/6,11/7,12/2,12/7,12/8,13/3,13/8,13/9,14/4,14/9,14/10,15/5,15/10,15/6}
				{\draw [ thick, blue, line width=0.15mm, opacity=0.9] (\i) -- (\j) ;}
				
				\foreach \i in {11,12,13,14,15}
				{
					\foreach \j in {1,2,3}
					{	
						\draw [ thick, red, line width=0.15mm, opacity=0.9] 
						(\i)--({(\j-2)*18+(\i-11)*72+90}:5);
					}		
				}
				
				\foreach \i in {6,7,8,9,10}
				{	
					\draw [ thick, red, line width=0.15mm, opacity=0.9] 
					(\i)--({(\i-6)*72+54}:5);
				}					
			\end{tikzpicture}
		}

	\end{minipage}} 
	\caption{A positive triangulation of negative girth 5}
	\label{fig:girth5}	
\end{figure}
  
  \FloatBarrier
  
\section{Remarks}
 
 We conclude with the following questions.
  
  \begin{problem}
  	Given an integer $k$, $k\geq 3$, what is the order of smallest signed graph whose negative girth is at least $k$ and whose circular chromatic number is at least 6. 
  \end{problem}
  
  \begin{problem}
  	What is the complexity of the following decision problem.\\
  	Input: A signed simple graph $\widehat{G}$ on the projective plane.\\
  	Output: YES if $\chi_c(G)=6$, NO otherwise (i.e., $\chi_c(G)< 6$). 
  \end{problem}

\textbf{Acknowledgments}~~~ The first part of the work has greatly benefited from discussion with G\'{a}bor Tardos and late G\'{a}bor Simonyi and the first two authors during their visit of 2023. This work has received support under the program ``Investissement d'Avenir" launched by the French Government and implemented by ANR, with the reference ``ANR‐18‐IdEx‐0001" as part of its program ``Emergence". Lujia Wang is partially supported by NSFC Grant No. 12371359.

\bibliographystyle{plain}
\bibliography{references}

\end{document}